\documentclass[11pt]{article}

\usepackage[T1]{fontenc}
\usepackage[english]{babel}
\usepackage[a4paper,margin=1in]{geometry}
\usepackage{amsmath,amssymb,amsthm,mathtools}
\usepackage{microtype}
\usepackage{hyperref}
\usepackage{tikz}
\usepackage{float}
\usepackage{enumitem}
\usepackage{cleveref}
\allowdisplaybreaks
\hypersetup{
  hidelinks,
  pdftitle={Intersections of Oriented Graphs and Tournaments},
  pdfauthor={Zhanping Yang and Qinghou Zeng},
  pdfsubject={Intersection discrepancy for oriented graphs and tournaments},
  pdfkeywords={discrepancy, oriented graphs, tournaments, graph intersections, transpositions}
}

\newtheorem{theorem}{Theorem}[section]
\newtheorem{lemma}[theorem]{Lemma}

\newtheorem{proposition}[theorem]{Proposition}
\newtheorem{observation}[theorem]{Observation}
\newtheorem{claim}[theorem]{Claim}
\newtheorem{problem}[theorem]{Problem}

\theoremstyle{remark}

\newcommand{\E}{\mathbb{E}}

\newcommand{\disc}{\operatorname{disc}}
\newcommand{\R}{\mathbb{R}}
\newcommand{\TT}{\mathrm{TT}}

\newcommand{\1}{\mathbf{1}}

\title{Intersections of Oriented Graphs and Tournaments}
\author{
  Zhanping Yang\thanks{Center for Discrete Mathematics, Fuzhou University, Fuzhou, Fujian 350108, China. Email: \texttt{yangzp@163.com}.}
  \and
  Qinghou Zeng\thanks{Corresponding author. Center for Discrete Mathematics, Fuzhou University, Fuzhou, Fujian 350108, China. Research supported by National Key R\&D Program of China (Grant No. 2023YFA1010202) and National Natural Science Foundation of China (Grant No. 12371342).  Email: \texttt{zengqh@fzu.edu.cn}.}
}
\date{}

\begin{document}
\maketitle

\begin{abstract}
Given two tournaments of order $n$, Bollob\'as and Scott defined their discrepancy as the largest deviation of their overlap from its random average under relabelling, and they asked whether the resulting discrepancy is always $\Omega (n^{3/2})$. We answer this question by proving that there is an absolute constant $c>0$ such that every pair of tournaments $T,U$ of order $n$ has discrepancy at least $cn^{3/2}$.

More generally, if $D$ and $H$ are oriented graphs of order \( n \) with \( e(D) = p \binom{n}{2} \) and \( e(H) = q \binom{n}{2} \) satisfying $16/n \leq p, q \leq 1 - 16/n$, then  there is an absolute constant \( c > 0 \) such that their discrepancy is at least $c(p(1-p)q(1-q))^{3}n^{3/2}$. We also show that these two-graph estimates extend to intersections of any fixed number of graphs, tournaments, and oriented graphs.
\end{abstract}



\section{Introduction}\label{sec:introduction}
An oriented graph is a digraph with no loops and no pair of oppositely directed arcs; a tournament is an orientation of a complete graph. Let $D$ and $H$ be oriented graphs of order $n$, with $e(D)=p\binom{n}{2}$ and $e(H)=q\binom{n}{2}$. If a uniformly random bijection is used to place $D$ and $H$ on a common vertex set, then each arc of $H$ belongs to the relabelled copy of $D$ with probability $p/2$. Hence their expected number of common arcs is $\frac{pq}{2}\binom{n}{2}$. For a bijection $\pi$ between the vertex sets, let $E_\pi(D)$ denote the relabelled arc set of $D$. Define
\begin{align}
	\disc^+(D,H)
	&=\max_{\pi} |E_\pi(D) \cap E(H)|-\frac{pq}{2}\binom n2,\label{eq:disc-plus-def}\\
	\disc^-(D,H)
	&=\frac{pq}{2}\binom n2-\min_{\pi} |E_\pi(D) \cap E(H)|,\label{eq:disc-minus-def}
\end{align}
and
\begin{equation*}
\disc(D,H)=\max\{\disc^+(D,H),\disc^-(D,H)\}.
\end{equation*}
	Thus $\disc^+(D,H)$ measures how much more than the random average the two oriented graphs can be made to agree, whereas $\disc^-(D,H)$ measures how much less they can be made to agree.

The study of discrepancy in graphs and related set systems originates in work of Erd\H{o}s and Spencer \cite{ErdosSpencer1972} on imbalances in colorings and in the graph-partition results of Erd\H{o}s, Goldberg, Pach and Spencer \cite{ErdosGoldbergPachSpencer1988}. Signed versions of discrepancy, together with extensions to graphs and hypergraphs, were developed systematically by Bollob\'as and Scott \cite{BollobasScottDiscrepancy2006}. Bollob\'as and Scott \cite{BollobasScottGraphs2011,BollobasScottHypergraphs2015} studied intersection discrepancy in graphs and hypergraphs, establishing a product lower bound for graphs under moderate-density assumptions and showing that additional structural hypotheses are needed in the hypergraph setting. Bollob\'as and Scott \cite{BollobasScottRandom2015} also proved that two independent random tournaments typically have discrepancy of order $n^{3/2}\sqrt{\log n}$ and obtained parallel results for random hypergraphs. Independently, Ma, Naves and Sudakov \cite{MaNavesSudakov2015} established closely related estimates for random graphs and hypergraphs. More recently, R\"aty, Sudakov and Tomon~\cite{RatySudakovTomon}
studied the positive discrepancy of graphs, obtaining density-dependent
lower bounds and establishing connections with MaxCut and spectral
parameters.

The tournament case already contains a natural extremal obstruction. Let $\TT_n$ be the transitive tournament on $[n]$, with $i\to j$ whenever $i<j$. Relabelling a tournament against $\TT_n$ is equivalent to choosing a linear ordering and counting the arcs directed forward in that ordering. Spencer \cite{Spencer1971,Spencer1980} proved that
\begin{equation*}
\min_{|T|=n}\disc^+(T,\TT_n)=\Theta(n^{3/2}).
\end{equation*}
 Since reversing the order of $\TT_n$ reverses every arc, one also has $\disc^+(T,\TT_n)=\disc^-(T,\TT_n)$. Thus $n^{3/2}$ is the smallest possible uniform scale already when one member of the pair is transitive.
 
 Motivated by the graph-intersection product theorem and by Spencer's result, Bollob\'as and Scott posed the following questions.
\begin{problem}[\textbf{Bollob\'as and Scott} \cite{BollobasScottHypergraphs2015,BollobasScottRandom2015}]\label{prob:tournaments}
Is there an absolute constant $c>0$ such that every pair $T,U$ of tournaments of order $n$ satisfies $\disc(T,U)\ge cn^{3/2}$?
\end{problem}
\begin{problem}[\textbf{Bollob\'as and Scott} \cite{BollobasScottHypergraphs2015}]
What can be said about the discrepancy of two oriented graphs of the same order?
\end{problem}
Our first main result answers Problem~\ref{prob:tournaments}.
\begin{theorem}\label{thm:tournament-intro}
	There is an absolute constant $c_1>0$ such that, for every $n\ge2$ and every pair of tournaments $T,U$ of order $n$,
	\begin{equation*}
	\disc^+(T,U)\disc^-(T,U)\ge c_1n^3.
	\end{equation*}
	Consequently,
	\begin{equation*}
	\disc(T,U)\ge \sqrt{c_1}\,n^{3/2}.
	\end{equation*}
\end{theorem}
The exponent $3/2$ is optimal, by Spencer's theorem. For general oriented graphs we obtain a quantitative result under the same moderate-density condition that appears in the graph-intersection theorem.
\begin{theorem}\label{thm:moderate-density}
There is an absolute constant $c_2>0$ with the following property. Let $D,H$ be oriented graphs of order $n$ such that $e(D)=p\binom n2$ and $e(H)=q\binom n2$, where ${16}/{n}\le p,q\le1-{16}/{n}$. Then
\begin{equation*}
\disc^+(D,H)\disc^-(D,H)\ge c_2 \left(p(1-p)q(1-q)\right)^6n^3.
\end{equation*}
Consequently,
\begin{equation*}
\disc(D,H)\ge \sqrt{c_2}\left(p(1-p)q(1-q)\right)^3n^{3/2}.
\end{equation*}
\end{theorem}
We also record fixed-family extensions of the two-graph results. Let $r\ge 2$, and let $G_1,\ldots,G_r$ be graphs of order $n$. For $1\le i\le r$, write $e(G_i)=p_i\binom{n}{2}$ and $0\le p_i\le 1$. For $\boldsymbol{\pi}=(\pi_1,\ldots,\pi_r)\in S_n^r$, let $G_i^{\pi_i}$ denote the graph obtained from $G_i$ by the relabelling
$\pi_i$, and define
\begin{equation*}
\operatorname{disc}_r^+(G_1,\ldots,G_r)=\max_{\boldsymbol{\pi}\in S_n^r} \left|
\bigcap_{i=1}^r E_{\pi_i}(G_i)\right|-\binom{n}{2}\prod_{i=1}^r p_i,
\end{equation*}
and
\begin{equation*}
\operatorname{disc}_r^-(G_1,\ldots,G_r)=\binom{n}{2}\prod_{i=1}^r p_i-\min_{\boldsymbol{\pi}\in S_n^r}\left|\bigcap_{i=1}^r E_{\pi_i}(G_i)\right|.
\end{equation*}
For $r=2$, Bollob\'as and Scott \cite{BollobasScottGraphs2011} proved the following theorem.
\begin{theorem}[\textbf{Bollob\'as and Scott} \cite{BollobasScottGraphs2011}]\label{thm:jgt2011}
Let \( G \) and \( H \) be graphs of order \( n \), and suppose that \( e(G) = p\binom{n}{2} \) and  \( e(H) = q\binom{n}{2} \), where \( 16/n \leq p, q \leq 1 - 16/n \). Then  
\begin{equation*}
\operatorname{disc}^+(G,H)\operatorname{disc}^-(G,H) \geq p^4(1-p)^4q^4(1-q)^4n^3/10^{20}.
\end{equation*}
\end{theorem}
The same estimate extends to any fixed number of graphs by conditioning on two relabellings and averaging over the others.
\begin{theorem}\label{thm:somegraphs}
Let $r\ge 2$, and let $G_1,\ldots,G_r$ be graphs of order $n$ with $e(G_i)=p_i\binom{n}{2}$, $1\le i\le r$. Suppose that there exist indices $1\le i<j\le r$ such that ${16}/{n}\le p_i,p_j\le 1-{16}/{n}$. Then
\begin{equation*}
\operatorname{disc}_r^+(G_1,\ldots,G_r)\operatorname{disc}_r^-(G_1,\ldots,G_r)\ge\frac{n^3}{10^{20}}\max_{i<j}\left\{\left(\prod_{\ell\ne i,j}p_\ell\right)^2
\bigl(
p_i(1-p_i)p_j(1-p_j)
\bigr)^4
\right\}.
\end{equation*}
\end{theorem}
We next define the analogous discrepancy quantities for several oriented
graphs. Let $r\ge 2$, and let $D_1,\ldots,D_r$ be oriented graphs of
order $n$. For \( 1 \leq i \leq r \), write \( e(D_i) = p_i \binom{n}{2} \), where \( 0 \leq p_i \leq 1 \), and identify all vertex sets with \([n]\). The intersections below are intersections of arc sets. Define
\begin{equation*}
\operatorname{disc}_r^+(D_1,\ldots,D_r)=
\max_{\boldsymbol{\pi}\in S_n^r}
\left|\bigcap_{i=1}^r E_{\pi_i}(D_i)\right|-\frac{1}{2^{r-1}}\binom{n}{2}\prod_{i=1}^r p_i,
\end{equation*}
and
\begin{equation*}
	\operatorname{disc}_r^-(D_1,\ldots,D_r)=\frac{1}{2^{r-1}}\binom{n}{2}\prod_{i=1}^r p_i-\min_{\boldsymbol{\pi}\in S_n^r}\left|\bigcap_{i=1}^r E_{\pi_i}(D_i)\right|.
\end{equation*}
For $r=2$, these definitions agree with (\ref{eq:disc-plus-def}) and (\ref{eq:disc-minus-def}). We obtain the following extensions.
\begin{theorem}\label{thm:sometournaments}
	Let $r\ge 2$ and $n\geq2$, and let $T_1,\ldots,T_r$ be tournaments of order $n$. Then
	\[
	\operatorname{disc}_r^+(T_1,\ldots,T_r)
	\operatorname{disc}_r^-(T_1,\ldots,T_r)
	\ge
	\frac{c_1}{4^{r-2}}\,n^3,
	\]
	where $c_1>0$ is the absolute constant in Theorem~\ref{thm:tournament-intro}.
\end{theorem}
\begin{theorem}\label{thm:someoriented}
	Let $r\ge 2$, and let $D_1,\ldots,D_r$ be oriented graphs of order $n$
	with $e(D_i)=p_i\binom{n}{2}$, $1\le i\le r$. Suppose that there exist $1\le i<j\le r$ such that ${16}/{n}\le p_i,p_j\le 1-{16}/{n}$. Then
	\begin{equation*}
	\operatorname{disc}_r^+(D_1,\ldots,D_r)
	\operatorname{disc}_r^-(D_1,\ldots,D_r)\ge c_2 n^3\max_{i<j}
	\left\{\left(\prod_{\ell\ne i,j}\frac{p_\ell}{2}\right)^2\bigl(p_i(1-p_i)p_j(1-p_j)
	\bigr)^6\right\},
	\end{equation*}
	where $c_2>0$ is the absolute constant in Theorem~\ref{thm:moderate-density}. 
\end{theorem}
The paper is organized as follows. Section~\ref{sec:normalization} introduces the centered matrix formulation and proves the random permutation-sum estimate used later. Section~\ref{sec:amplification} develops the transposition-amplification argument. Section~\ref{sec:local-identities} gives the orthogonal structural decomposition and the exact four-vertex support and triangle-circulation identities. Section~\ref{sec:structure-energy} derives the row-difference formula and the bound for discrepancy. Section~\ref{sec:mainresults} completes the proofs of Theorems~\ref{thm:tournament-intro} and \ref{thm:moderate-density}. Section~\ref{sec:somegraphs} uses a conditioning argument to reduce the discrepancy of several graphs to that of two graphs and proves Theorems~\ref{thm:somegraphs}--\ref{thm:someoriented}.
\section{Definitions and auxiliary tools}\label{sec:normalization}
We will assume that $D$ and $H$ are oriented graphs of order $n$ with vertex set $V=[n]$. We define an action of the symmetric group $S_n$ on $D$ by 
\begin{equation*}
D_\pi=(V,E_\pi(D)),
\end{equation*}
 where 
 \begin{equation*}
  E_\pi(D)=\{i\to j:\pi(i)\to \pi(j)\in E(D)\}.
 \end{equation*}
 We identify \( D \) with the matrix $A=(A_{ij})_{n\times n}$ with entries
 \begin{equation}\label{eq:Aij}
 	A_{ij} =
 	\begin{cases} 
 		\mathbf{1}_{\{i \to j \in E(D)\}} - p/2, & i \neq j, \\
 		0, & i = j,
 	\end{cases}
 \end{equation}
where \( \mathbf{1}_{\{i\to j\in E(D)\}} = 1 \) if \( i\to j \in E(D) \) and $0$ otherwise. Similarly, we identify $D_\pi$ with the corresponding matrix
$A_\pi$. For an oriented graph $H$ with corresponding matrix $B=(B_{ij})_{n\times n}$, we have
\begin{align}\label{eq:<B_pi,C>}
 \langle A_{\pi},B\rangle&=\sum_{i\ne j}A_{\pi(i)\pi(j)}B_{ij}\notag\\&=\sum_{i\ne j}\left(\textbf{1}_{\{\pi(i) \to \pi(j) \in E(D)\}} - \frac{p}{2}\right)\left(\textbf{1}_{\{i \to j \in E(H)\}} - \frac{q}{2}\right)\notag\\&=\sum_{i \neq j} \left(\textbf{1}_{\{i \to j \in E_\pi(D)\}} \textbf{1}_{\{i \to j \in E(H)\}}\right)-\frac{pq}{2}\binom n2\notag \\&=|E_\pi(D) \cap E(H)|-\frac{pq}{2}\binom n2,
\end{align}
Let $J$ denote the all-ones matrix. Then
\begin{equation}\label{eq:<A,I>}
	\langle A_{\pi},J\rangle=\langle A,J\rangle=\sum_{i\ne j}A_{ij}=\sum_{i\ne j}\left(\textbf{1}_{\{i \to j \in E(D)\}}\right) - p\binom{n}{2}=0.
\end{equation}
\begin{observation}\label{lem:random-center}
	If $\pi$ is chosen uniformly at random from $S_n$, then $\E_\pi\langle A_{\pi},B\rangle=0$.
\end{observation}
\begin{proof}[\bf Proof]
	For fixed $i\ne j$, the ordered pair $(\pi(i),\pi(j))$ is uniform over all $n(n-1)$ ordered pairs of distinct vertices. By (\ref{eq:<A,I>}), we have
	\begin{equation*}
		\E_\pi \langle A_\pi,B\rangle
		=\sum_{i\ne j}B_{ij}\,\E_\pi A_{\pi(i)\pi(j)}=\frac{1}{n(n-1)}\sum_{i\ne j}B_{ij}\sum_{x\ne y}A_{xy}=0.
	\end{equation*}
	This completes the proof.
\end{proof}
 Combining (\ref{eq:disc-plus-def}), (\ref{eq:disc-minus-def}) and (\ref{eq:<B_pi,C>}), we may write the positive discrepancy of $D$ with respect to $H$ by
 \begin{equation*}
 \disc^+(D,H)=\max_\pi \langle A_{\pi},B\rangle
 \end{equation*} 
 and the negative discrepancy of $D$ with respect to $H$
\begin{equation*}
\disc^-(D,H)=-\min_\pi \langle A_{\pi},B\rangle.
\end{equation*}

\begin{lemma}\label{lem:EYphi}
	Let $N,d\in\mathbb{N}$ with $N\ge 2$, and let $L>0$. Suppose that $\textbf{u}_1,\ldots,\textbf{u}_N,\textbf{v}_1,\ldots,\textbf{v}_N\in\mathbb R^d$ satisfy $\sum_i \textbf{u}_i=\sum_j \textbf{v}_j=\textbf{0}$ and $\|\textbf{u}_i\|,\|\textbf{v}_j\|\le L$. Set $U=\sum_i \textbf{u}_i \textbf{u}_i^{\mathsf T}$ and $V=\sum_j \textbf{v}_j \textbf{v}_j^{\mathsf T}$. If $\phi$ is a uniformly random permutation of $[N]$ and $Y=\sum_{i} \textbf{u}_i^{\mathsf{T}} \textbf{v}_{\phi(i)}$, then there is an absolute constant $c>0$ such that
	\[
	\mathbb E_{\phi}|Y|
	\ge
	c\,\frac{\operatorname{tr}(UV)^{3/2}}{L^4N^{5/2}}.
	\]
\end{lemma}
\begin{proof}[\bf Proof]
		For each $i$, $\phi(i)$ is uniform on $[N]$, hence $\mathbb{E}_\phi \boldsymbol{v}_{\phi(i)}=\boldsymbol{0}$ and $\mathbb{E}_\phi Y=0$. Moreover,
\begin{equation}\label{eq:V/N}
	\mathbb E_\phi\left(\boldsymbol{v}_{\phi(i)}\boldsymbol{v}_{\phi(i)}^{\mathsf T}\right)=\frac{1}{N} \sum_{a=1}^{N} \boldsymbol{v}_a \boldsymbol{v}_a^{\mathsf T}=\frac{V}{N}, 
\end{equation}
and
\begin{equation}\label{eq:0dd}
\left( \sum_a \boldsymbol{v}_a \right) \left( \sum_b \boldsymbol{v}_b \right)^{\mathsf T}=\sum_{a,b} \boldsymbol{v}_a \boldsymbol{v}_b^{\mathsf T} = \sum_a \boldsymbol{v}_a \boldsymbol{v}_a^{\mathsf T} + \sum_{a \neq b} \boldsymbol{v}_a \boldsymbol{v}_b^{\mathsf T}
=V+\sum_{a \neq b} \boldsymbol{v}_a \boldsymbol{v}_b^{\mathsf T} = \mathbf{0}_{d \times d}.
\end{equation}
Similarly, we have
\begin{equation}\label{eq:0ddV}
\left( \sum_a \boldsymbol{u}_a \right)^{\mathsf T} V \left( \sum_b \boldsymbol{u}_b \right) = \sum_{a,b} \boldsymbol{u}_a^{\mathsf T} V \boldsymbol{u}_b = \sum_a \boldsymbol{u}_a^{\mathsf T} V \boldsymbol{u}_a + \sum_{a \neq b} \boldsymbol{u}_a^{\mathsf T} V \boldsymbol{u}_b = 0.
\end{equation}
	For $i\ne k$, the ordered pair $(\phi(i),\phi(k))$ is uniformly
	distributed over all ordered pairs of distinct elements of $[N]$. By (\ref{eq:0dd}), we have
	\begin{equation}\label{eq:-V/NN-1}
	\mathbb E_\phi\left(\boldsymbol{v}_{\phi(i)}\boldsymbol{v}_{\phi(k)}^{\mathsf T}\right)=\frac{1}{N(N-1)}
	\sum_{a\ne b}\boldsymbol{v}_a\boldsymbol{v}_b^{\mathsf T}=-\frac{V}{N(N-1)}.
	\end{equation}
	Combining (\ref{eq:V/N}), (\ref{eq:0ddV}) and (\ref{eq:-V/NN-1}), we obtain
	\begin{align}\label{eq:EY2}
	\mathbb E_\phi Y^2&=\mathbb{E}_{\phi} \left( \sum_{i=1}^{N} \left( \boldsymbol{u}_i^{\mathsf T} \boldsymbol{v}_{\phi(i)} \right)^2 + \sum_{i \neq k} \left( \boldsymbol{u}_i^{\mathsf T} \boldsymbol{v}_{\phi(i)} \right) \left( \boldsymbol{u}_k^{\mathsf T} \boldsymbol{v}_{\phi(k)} \right) \right)\notag\\&=\mathbb{E}_{\phi} \left( \sum_{i=1}^{N} \boldsymbol{u}_i^{\mathsf T} \boldsymbol{v}_{\phi(i)} \boldsymbol{v}_{\phi(i)}^{\mathsf T} \boldsymbol{u}_i + \sum_{i \neq k} \boldsymbol{u}_i^{\mathsf T} \boldsymbol{v}_{\phi(i)} \boldsymbol{v}_{\phi(k)}^{\mathsf T} \boldsymbol{u}_k \right)\notag\\&=\sum_{i=1}^N \boldsymbol{u}_i^{\mathsf T} \mathbb{E}_\phi \left( \boldsymbol{v}_{\phi(i)} \boldsymbol{v}_{\phi(i)}^{\mathsf T} \right) \boldsymbol{u}_i + \sum_{i \neq k} \boldsymbol{u}_i^{\mathsf T} \mathbb{E}_\phi \left( \boldsymbol{v}_{\phi(i)} \boldsymbol{v}_{\phi(k)}^{\mathsf T} \right) \boldsymbol{u}_k\notag\\&=\frac{1}{N} \sum_{i=1}^N \boldsymbol{u}_i^{\mathsf T} V \boldsymbol{u}_i - \frac{1}{N(N-1)} \sum_{i \neq k} \boldsymbol{u}_i^{\mathsf T} V \boldsymbol{u}_k\notag\\&=\frac{1}{N} \sum_{i=1}^N \boldsymbol{u}_i^{\mathsf T} V \boldsymbol{u}_i + \frac{1}{N(N-1)} \sum_{i=1}^N \boldsymbol{u}_i^{\mathsf T} V \boldsymbol{u}_i\notag\\&=\frac{1}{N-1} \sum_{i=1}^N \boldsymbol{u}_i^{\mathsf T} V \boldsymbol{u}_i=\frac{\operatorname{tr}(UV)}{N-1}.
	\end{align}
For $0\le k\le N-1$, define $M_k=\mathbb E_\phi\left(Y\mid\phi(1),\ldots,\phi(k)\right)$ with $M_0=0$ and $M_{N-1}=Y$. Fix \( 1 \leq k \leq N - 1 \), and let \( R = [N] \setminus \{ \phi(1), \ldots, \phi(k-1) \} \). Then \( |R| = N - k + 1 \). For \( i > k \) and \( a \in R \), we have
\begin{equation*}
\mathbb E_\phi\left(\boldsymbol{v}_{\phi(i)}\,\middle|\,\phi(1),\ldots,\phi(k-1),\phi(k)=a\right)=\frac1{N-k}\sum_{c\in R\setminus\{a\}}\boldsymbol{v}_c,
\end{equation*}
and
\begin{equation*}
\mathbb E_\phi
\left(\sum_{i>k}\boldsymbol{u}_i^{\mathsf T} \boldsymbol{v}_{\phi(i)}
\,\middle|\,
\phi(1),\ldots,\phi(k-1),\phi(k)=a
\right)=
\frac1{N-k}
\left(\sum_{i>k}\boldsymbol{u}_i^{\mathsf T}\right)
\left(
\sum_{c\in R\setminus\{a\}}\boldsymbol{v}_c
\right).
\end{equation*}
Conditioning on $\phi(k)=a$, we have
\begin{equation*}
H(a)=\mathbb E_\phi\left(Y\,\middle|\,\phi(1),\ldots,\phi(k-1),\phi(k)=a\right)=\sum_{i<k}\boldsymbol{u}_i^{\mathsf T} \boldsymbol{v}_{\phi(i)}+\boldsymbol{u}_k^{\mathsf T}\boldsymbol{v}_a+\frac{\sum_{i>k}\boldsymbol{u}_i^{\mathsf T}}{N-k}\sum_{c\in R\setminus\{a\}}\boldsymbol{v}_c.
\end{equation*}
Since $\phi(k)$ is uniform on $R$ conditional on $\phi(1),\ldots,\phi(k-1)$, we have
\begin{equation*}
	M_{k-1}=\sum_{a\in R}\Pr\!\left(\phi(k)=a\,\middle|\,\phi(1),\ldots,\phi(k-1)\right)H(a)=
\frac1{|R|}\sum_{a\in R}H(a).
\end{equation*}
For $a,b\in R$,
\begin{equation*}
|H(a)-H(b)|=\left|\boldsymbol{u}_k^{\mathsf T}(\boldsymbol{v}_a-\boldsymbol{v}_b)+\frac{\sum_{i>k}\boldsymbol{u}^{\mathsf T}_i}{N-k}(\boldsymbol{v}_b-\boldsymbol{v}_a)\right|\leq\|\boldsymbol{v}_a-\boldsymbol{v}_b\|\cdot\left|\left|\boldsymbol{u}^{\mathsf T}_k-\frac{\sum_{i>k}\boldsymbol{u}^{\mathsf T}_i}{N-k}\right|\right|\leq 4L^2.
\end{equation*}
Thus, for each $k$ 
\begin{equation*}
|M_k-M_{k-1}|=\left|H(\phi(k))-\frac1{|R|}\sum_{a\in R}H(a)
\right|\le\max_{a,b\in R}|H(a)-H(b)|\le4L^2.
\end{equation*}
By the Azuma--Hoeffding inequality \cite{Azuma1967}, we have $\Pr(|Y|\ge t)\le2\exp\left(-\frac{t^2}{32L^4N}\right)$. Consequently, 
\begin{equation}\label{eq:EY4}
\mathbb E_\phi Y^4=4\int_0^\infty t^3\Pr(|Y|\ge t)\,dt\le8\int_0^\infty t^3
\exp\left(-\frac{t^2}{32L^4N}\right)dt=4096L^8N^2.
\end{equation}
By H\"older's inequality, we have $\mathbb E_\phi Y^2\le(\mathbb E_\phi|Y|)^{2/3}(\mathbb E_\phi |Y|^4)^{1/3}$. If \( \text{tr}(UV) = 0 \), the conclusion is immediate. Otherwise, \( \mathbb{E}_\phi Y^4 > 0 \), and hence
\begin{equation*}
\mathbb E_\phi|Y|\ge
\frac{(\mathbb E_\phi Y^2)^{3/2}}
{(\mathbb E_\phi Y^4)^{1/2}}\ge
\frac{\operatorname{tr}(UV)^{3/2}}
{64L^4N(N-1)^{3/2}}\ge
\frac1{64}\cdot
\frac{\operatorname{tr}(UV)^{3/2}}
{L^4N^{5/2}}.
\end{equation*}
Combining (\ref{eq:EY2}) and (\ref{eq:EY4}) completes the proof of the lemma, with $c=1/64$.
\end{proof}
\begin{proposition}\label{prop:constant-shift}
	If $\E Z=0$, then for every constant $a\in\R$, $\E|Z+a|\ge\E|Z|/2$.
\end{proposition}
\begin{proof}[\bf Proof]
	By Jensen's inequality, we have $|a|=|\E(Z+a)|\le\E|Z+a|$. Also $|Z|\le|Z+a|+|a|$. Taking expectations proves the proposition.
\end{proof}
\section{Bounding in terms of transpositions}\label{sec:amplification}

This section translates a local effect of a single transposition into a global product lower bound for the positive and negative extrema. This adapts the transposition-amplification argument of Bollob\'as and Scott \cite{BollobasScottHypergraphs2015} to oriented graphs.

Fix the transposition $\tau = (xy)$. Let $A_\rho$ and $B_\sigma$ be independent uniformly random relabellings of $A$ and $B$. Define
\begin{equation}\label{eq:gamma-def}
	\gamma(A,B)=
	\E_{\rho,\sigma}\left|\langle (A_{\rho})_{\tau},B_\sigma\rangle-\langle A_\rho,B_\sigma\rangle\right|.
\end{equation}
Let $\tau_i=(x_iy_i)$, $i\in I$, be pairwise disjoint
transpositions. For $J\subseteq I$, write $\tau_J=\prod_{i\in J}\tau_i$.
\begin{lemma}\label{lem:quadratic}
	Let $C,D$ be fixed real matrices with zero diagonal. Choose a random subset $J\subseteq I$
	by including each $i\in I$ independently with probability $p_\theta$. Then
	\begin{equation*}
		\E_J\bigl(\langle C_{\tau_J},D\rangle-\langle C,D\rangle\bigr)
		=p_\theta\sum_{i\in I}\delta_i+p_\theta^2\sum_{i<j}\beta_{ij},
	\end{equation*}
	where $\delta_i=\langle C_{\tau_i},D\rangle-\langle C,D\rangle$ and $\beta_{ij}=\langle C_{\tau_i\tau_j},D\rangle-\langle C_{\tau_i},D\rangle-\langle C_{\tau_j},D\rangle+\langle C,D\rangle$.
\end{lemma}

\begin{proof}[\bf Proof]
	For each $i \in I$, let $\xi_i = \mathbf{1}_{\{i \in J\}}$ be the indicator random variable and define $\Delta_J := \langle C_{\tau_J}, D \rangle - \langle C, D \rangle$. Then
	\begin{equation*}
		\Delta_J = \sum_{r \neq s} c_{\tau_J(r), \tau_J(s)} d_{rs} - \sum_{r \neq s} c_{rs} d_{rs} =  \sum_{r \neq s} d_{rs} \big( c_{\tau_J(r), \tau_J(s)} - c_{rs} \big).
	\end{equation*}
	Since the transpositions $\{\tau_i\}_{i \in I}$ are pairwise disjoint, for any fixed ordered pair $(r, s)$, the permuted entry $c_{\tau_J(r), \tau_J(s)}$ depends on at most two indicator variables $\xi_i$. We analyze this dependence by cases:
	\begin{itemize}
		\item If neither $r$ nor $s$ belongs to the support of any $\tau_i$, then $c_{\tau_J(r), \tau_J(s)} - c_{rs} = 0$.
		\item If exactly one endpoint, say $r$, lies in the support of some $\tau_i$, let
		$r'=\tau_i(r)$. Then $c_{\tau_J(r),\tau_J(s)}-c_{rs}=\xi_i(c_{r's}-c_{rs})$.
		\item If $r$ and $s$ both belong to the same transposition $\tau_i$, $s = \tau_i(r)$ and $r = \tau_i(s)$, then $c_{\tau_J(r), \tau_J(s)} - c_{rs} = \xi_i(c_{sr} - c_{rs})$.
		\item If $r$ and $s$ belong to different transpositions, say $\tau_i$ and $\tau_j$ respectively, let $r' = \tau_i(r)$ and $s' = \tau_j(s)$. Then $c_{\tau_J(r), \tau_J(s)} - c_{rs} = \xi_i(c_{r's} - c_{rs}) + \xi_j(c_{rs'} - c_{rs}) + \xi_i\xi_j(c_{r's'} - c_{r's} - c_{rs'} + c_{rs})$.
	\end{itemize}
	Thus
	\[
	\bigl(\langle C_{\tau_J},D\rangle-\langle C,D\rangle\bigr)=\sum_i a_i\xi_i+\sum_{i<j}a_{ij}\xi_i\xi_j,
	\]
	where $a_i=\delta_i$ and
	$a_{ij}=\beta_{ij}$. Taking expectations and using
	$\E\xi_i=p_\theta$ and $\E\xi_i\xi_j=p_\theta^2$ proves the lemma.
\end{proof}

\begin{lemma}\label{lem:amplification}
There is an absolute constant $c_0>0$ such that, for every pair of
oriented graphs $D$ and $H$ of order $n$, with corresponding matrices
$A$ and $B$ defined as in (\ref{eq:Aij}). Then
	\begin{equation*}
		\disc^+(D,H)\disc^-(D,H)\ge c_0 n^2\gamma(A,B)^2,
	\end{equation*}
	where $c_0=10^{-4}$.
\end{lemma}

\begin{proof}[\bf Proof]
	The assertion is trivial if $\gamma(A, B) = 0$. We may assume without loss of generality that $\operatorname{disc}^+(D, H) \le \operatorname{disc}^-(D, H)$. If $\operatorname{disc}^+(D, H) \ge n\gamma(A,B)/100$, then  we are done. We may therefore assume that
	\begin{equation}\label{eq:disc_alpha_def}
		\operatorname{disc}^+(D, H) = \frac{n\gamma(A,B)}{100\alpha},
	\end{equation}
	for some constant $\alpha > 1$.
	
	Let $t = \lfloor n/2 \rfloor$, and fix $t$ pairwise disjoint transpositions $\tau_1, \ldots, \tau_t$. Let $A_\rho, B_\sigma$ be independent, uniformly random relabellings, and set $\delta_i = \langle (A_{\rho})_{\tau_i}, B_\sigma \rangle - \langle A_\rho, B_\sigma \rangle$. Since applying a fixed transposition preserves the uniform distribution of $A^\rho$, we have $\mathbb{E}_{\rho,\sigma} \delta_i = 0$ and $\mathbb{E}_{\rho,\sigma}|\delta_i| = \gamma(A,B)$.
	
	Let $I^+ = \{i : \delta_i > 0\}$ and $\Delta^+ = \sum_{i=1}^t (\delta_i)^+$. Since $z^+ = (|z| + z)/2$, we have $\mathbb{E}_{\rho,\sigma}\Delta^+ = t\gamma(A,B)/2$. By Observation \ref{lem:random-center}, we have $\mathbb{E}_{\rho,\sigma} \langle A_\rho, B_\sigma \rangle = 0$, and so by linearity of expectation,
	\begin{equation*}
		\mathbb{E}_{\rho,\sigma} \big(\alpha \langle A_\rho, B_\sigma \rangle + \Delta^+ \big)= \frac{t}{2}\gamma(A,B).
	\end{equation*}
	We can therefore choose \(\rho\), \(\sigma\) such that \(\alpha \langle A_{\rho}, B_{\sigma} \rangle + \Delta_{A_{\rho}, B_{\sigma}}(I^+) \geq t \gamma(A,B )/2\). Replacing \(A\) and \(B\) by \(A_{\rho}\) and \(B_{\sigma}\), we may therefore assume that
	\begin{equation}\label{eq:alpha_X_bound}
		\alpha \langle A, B\rangle + \Delta^+ \ge \frac{t}{2}\gamma(A,B).
	\end{equation}
	Now consider the effects of applying $\tau_J$, where $J \subseteq I^+$ is a random subset of $I^+$ formed by including each $i \in I^+$ independently with probability $p_\theta$. Consider the random variable $ \langle A_{\tau_J}, B \rangle$. Lemma \ref{lem:quadratic} tells us that
	\begin{equation*}
		\mathbb{E}_J(\langle A_{\tau_J}, B \rangle - \langle A, B \rangle) = p_\theta\Delta^+ + A_2 p_\theta^2,
	\end{equation*}
	for some real number $A_2$. Taking $p_\theta = 1/\alpha$ and using \eqref{eq:alpha_X_bound}, we obtain
	\begin{equation}\label{eq:XJ_lower_bound}
		\mathbb{E}_J \langle A_{\tau_J}, B \rangle = \langle A, B \rangle + \frac{\Delta^+}{\alpha} + \frac{A_2}{\alpha^2} \ge \frac{t\gamma(A,B)}{2\alpha} + \frac{A_2}{\alpha^2}.
	\end{equation}
	Equation \eqref{eq:disc_alpha_def} implies that $\langle A^{\tau_J}, B \rangle \le n\gamma(A,B)/100\alpha$ for any choice of $J$. Hence, equation \eqref{eq:XJ_lower_bound} implies that
	\begin{equation*}
		\frac{A_2}{\alpha^2} \le \frac{n\gamma(A,B)}{100\alpha} - \frac{t\gamma(A,B)}{2\alpha}.
	\end{equation*}
	Since $t \ge n/3$ for $n \ge 2$, we have $A_2/\alpha^2 \le -n\gamma(A,B)/7\alpha$, which yields
	\begin{equation}\label{eq:A2_upper_bound}
		A_2 \le -\frac{1}{7}\alpha n\gamma(A,B).
	\end{equation}
	Let $q(p_\theta) = p_\theta\Delta^+ + A_2 p_\theta^2$. Since $q(1) - 2q(1/2) = A_2/2$, we must have $\max\{|q(1)|, |q(1/2)|\} \ge |A_2|/6$. Combining this with \eqref{eq:A2_upper_bound}, we get
	\begin{equation}\label{eq:poly_max}
		\max\{|q(1)|, |q(1/2)|\} \ge \frac{\alpha n\gamma(A,B)}{42}.
	\end{equation}
	Choose $p_\theta \in \{1/2, 1\}$ for which \eqref{eq:poly_max} holds. By Jensen's inequality, we have
	\begin{equation*}
		\mathbb{E}_J|\langle A_{\tau_J}, B \rangle - \langle A, B \rangle| \ge |\mathbb{E}_J(\langle A_{\tau_J}, B \rangle - \langle A, B \rangle)| \ge \frac{\alpha n\gamma(A,B)}{42}.
	\end{equation*}
	For every $J$, both $\langle A, B \rangle$ and $\langle A_{\tau_J}, B \rangle$ are at most $\operatorname{disc}^+(D, H)$. Since $\alpha > 1$, we have $1/100\alpha^2 - 1/42 \le 1/100 - 1/42 = -29/2100$. Hence
	\begin{align*}
		\mathbb{E}_J \min\{\langle A_{\tau_J}, B \rangle,\langle A, B \rangle\} &= \mathbb{E}_J \big(\max\{\langle A_{\tau_J}, B \rangle,\langle A, B \rangle\} - |\langle A, B \rangle - \langle A_{\tau_J}, B \rangle| \big) \\
		&\le \operatorname{disc}^+(D, H) - \mathbb{E}_J|\langle A_{\tau_J}, B \rangle - \langle A, B \rangle| \\
		&\le \frac{n\gamma(A,B)}{100\alpha} - \frac{\alpha n\gamma(A,B)}{42} \\
		&= \alpha n \gamma(A,B) \left( \frac{1}{100\alpha^2} - \frac{1}{42} \right)\\&\le -\frac{29}{2100} \alpha n \gamma(A,B).
	\end{align*}
	In particular, there is some $J$ such that one of $\langle A, B \rangle$ and $\langle A_{\tau_J}, B \rangle$ is at most $-29\alpha n\gamma(A,B)/2100$, and so $\operatorname{disc}^-(D, H) \ge 29\alpha n\gamma(A,B)/2100$. By \eqref{eq:disc_alpha_def}, we conclude
	\begin{equation*}
		\operatorname{disc}^+(D, H) \operatorname{disc}^-(D, H) \ge \left(\frac{n\gamma(A,B)}{100\alpha}\right) \left(\frac{29\alpha n\gamma(A,B)}{2100}\right) \ge 10^{-4} n^2 \gamma(A,B)^2.
	\end{equation*}
	This completes the proof.
\end{proof}
\section{Four-vertex support and triangle circulation}\label{sec:local-identities}
For an oriented graph $D$, we define the symmetric and skew-symmetric parts of $A$ by
\begin{equation}\label{eq:SDKD}
S_D=\frac{A+A^{\mathsf T}}2,
\qquad
K_D=\frac{A-A^{\mathsf T}}2.
\end{equation}
Let $r_D=S_D\1$ and $t_D=K_D\1$, where $r_D,t_D\in \mathbb{R}^n$ and $\textbf{1}=(1,1,\ldots,1)^{\mathsf T}\in \mathbb{R}^n$. If $i\ne j$, then we define 
\begin{equation}\label{eq:SD1KD1}
(S_D^{(1)})_{ij}=\frac{r_D(i)+r_D(j)}{n-2},
\qquad
(K_D^{(1)})_{ij}=\frac{t_D(i)-t_D(j)}{n},
\end{equation}
and set $(S_D^{(1)})_{ii}=(K_D^{(1)})_{ii}=0$. Put $R_D=S_D-S_D^{(1)}$, $C_D=K_D-K_D^{(1)}$. Define $S_H^{(1)}$, $K_H^{(1)}$, $R_H$ and $C_H$ analogously.
\begin{lemma}\label{lem:orthogonal-decomposition}
	Let $n\ge3$. Then $A=S_D^{(1)}+K_D^{(1)}+R_D+C_D$ and the four displayed components are pairwise orthogonal. Moreover, for every permutation $\pi$,
	\begin{equation}\label{eq:first-order-correlation}
		\left\langle(S_D^{(1)}+K_D^{(1)})_\pi,
		S_H^{(1)}+K_H^{(1)}
		\right\rangle=
		\frac{2}{n-2}\sum_i r_D(\pi(i))r_H(i)+\frac{2}{n}\sum_i t_D(\pi(i))t_H(i).
	\end{equation}
\end{lemma}
\begin{proof}[\bf Proof]
	By definition, we have
	\begin{equation*}
		S_D^{(1)}+K_D^{(1)}+R_D+C_D=S_D+K_D=A.
	\end{equation*}
	For each $i$,
	\begin{align*}
		r_D(i)&=\sum_{j\ne i}(S_D)_{ij}=\sum_{j\ne i}\frac{A_{ij}+A_{ji}}2=\frac{1}{2}\sum_{j\neq i}\left(\textbf{1}_{\{i \to j \in E(D)\}}+\textbf{1}_{\{j \to i \in E(D)\}}-p\right)\\&=\frac{d_D^+(i)+d_D^-(i)}2-\frac{p}{2}(n-1)=\frac{1}{2}\left(d_D^+(i)+d_D^-(i)-\frac{2|E(D)|}{n}\right),
	\end{align*}
	and
	\begin{align*}
		t_D(i)&=\sum_{j\ne i}{(K_D)}_{ij}=\sum_{j\ne i}\frac{A_{ij}-A_{ji}}2=\frac{1}{2}\sum_{j\neq i}\left(\textbf{1}_{\{i \to j \in E(D)\}}-\textbf{1}_{\{j \to i \in E(D)\}}\right)=\frac{1}{2}\left(d_D^+(i)-d_D^-(i)\right).
	\end{align*}
Thus, $r_D$ records the deviation of the vertex degrees in the underlying
graph from their average, while $t_D$ records the imbalance between the
out-degrees and in-degrees. By (\ref{eq:<A,I>}), we have
	\begin{equation*}
		\sum_i r_D(i)=\sum_i\sum_{j\ne i}\frac{A_{ij}+A_{ji}}2=\frac{\langle A,J\rangle+\langle A,J\rangle}{2}=0,
	\end{equation*}
	Similarly, we have $\sum_i t_D(i)=\left(\langle A,J\rangle-\langle A,J\rangle\right)/2=0$. Hence
	\[
	\sum_{j\ne i}(S_D^{(1)})_{ij}=\sum_{j\ne i}\frac{r_D(i)+r_D(j)}{n-2}
	=\frac{(n-1)r_D(i)+\sum_{j\ne i}r_D(j)}{n-2}=\frac{(n-2)r_D(i)+\sum_{i}r_D(i)}{n-2}=r_D(i),
	\]
	and similarly
	\[
	\sum_{j\ne i}(K_D^{(1)})_{ij}=\sum_{j\ne i}\frac{t_D(i)-t_D(j)}{n}
	=\frac{(n-1)t_D(i)-\sum_{j\ne i}t_D(j)}{n}=\frac{nt_D(i)-\sum_{i}t_D(i)}{n}=t_D(i).
	\]
	Thus 
	\begin{equation}\label{eq:RD}
		\sum_{j\ne i}(R_D)_{ij}=\sum_{j\ne i}\left((S_D)_{ij}-(S_D^{(1)})_{ij}\right)=\sum_{j\ne i}(S_D)_{ij}-\sum_{j\ne i}(S_D^{(1)})_{ij}=0,
	\end{equation}
	and
	\begin{equation}\label{eq:CD}
		\sum_{j\ne i}(C_D)_{ij}=\sum_{j\ne i}\left((K_D)_{ij}-(K_D^{(1)})_{ij}\right)=\sum_{j\ne i}(K_D)_{ij}-\sum_{j\ne i}(K_D^{(1)})_{ij}=0.
	\end{equation}
	Since symmetric matrices are orthogonal to skew-symmetric matrices, we have $\langle S_D^{(1)},K_D^{(1)}\rangle=\langle S_D^{(1)},C_D\rangle=\langle R_D,K_D^{(1)}\rangle=\langle R_D,C_D\rangle=0$. Since $R_D$ is symmetric, (\ref{eq:RD}) gives
	\begin{align*}
		\langle S_D^{(1)},R_D\rangle&=\frac{1}{n-2}\sum_{i\neq j}\left((r_D(i)+r_D(j))(R_D)_{ij}\right)\\&=\frac{1}{n-2}\sum_{i\neq j}\left(r_D(i)(R_D)_{ij}+r_D(j)(R_D)_{ij}\right)\\&=\frac{1}{n-2}\left(\sum_{i\neq j} r_D(i)(R_D)_{ij}+\sum_{i\neq j}r_D(j)(R_D)_{ji}\right)\\&=\frac{2}{n-2}\sum_i r_D(i)\sum_{j\ne i}(R_D)_{ij}=0.
	\end{align*}
	Similarly $C_D$ is a skew-symmetric, and (\ref{eq:CD}) gives
	\[
	\langle K_D^{(1)},C_D\rangle=\frac{1}{n}\sum_{i\neq j}\left((t_D(i)-t_D(j))(C_D)_{ij}\right)
	=\frac{2}{n}\sum_i t_D(i)\sum_{j\ne i}(C_D)_{ij}=0.
	\]
	This proves pairwise orthogonality. By the definition of the permutation action,
	\[
	\bigl((S_D^{(1)})_\pi\bigr)_{ij}
	=
	\frac{r_D(\pi(i))+r_D(\pi(j))}{n-2}.
	\]
	Since $\pi$ is a permutation and
	$\sum_i r_D(i)=\sum_i r_H(i)=0$, we have
	$\sum_i r_D(\pi(i))=0$. Hence
	\begin{align*}
		\left\langle (S_D^{(1)})_\pi,S_H^{(1)}\right\rangle&=\frac{1}{(n-2)^2}\sum_{i\ne j}(r_D(\pi(i))+r_D(\pi(j)))(r_H(i)+r_H(j))
		\\&=\frac{1}{(n-2)^2}\sum_{i\ne j}\left(r_D(\pi(i))r_H(i)+r_D(\pi(i))r_H(j)+r_D(\pi(j))r_H(i)+r_D(\pi(j))r_H(j)\right)\\&=\frac{1}{(n-2)^2}\left(2(n-1)\sum_{i}r_D(\pi(i))r_H(i)+\sum_{i\ne j}r_D(\pi(i))r_H(j)+\sum_{i\ne j}r_D(\pi(j))r_H(i)\right)\\&=\frac{1}{(n-2)^2}\left(2(n-1)\sum_{i}r_D(\pi(i))r_H(i)+2\sum_i r_D(\pi(i))\sum_{j\ne i}r_H(j)\right)\\&=\frac{1}{(n-2)^2}\left(2(n-1)\sum_{i}r_D(\pi(i))r_H(i)-2\sum_i r_D(\pi(i))r_H(i)\right)\\&=\frac{2}{n-2}\sum_i r_D(\pi(i))r_H(i).
	\end{align*}
	Similarly, we have $\sum_i t_D(\pi(i))=\sum_i t_H(i)=0$. Hence
	\begin{align*}
		\left\langle (K_D^{(1)})_\pi,K_H^{(1)}\right\rangle&=\frac{1}{n^2}\sum_{i\ne j}(t_D(\pi(i))-t_D(\pi(j)))(t_H(i)-t_H(j))\\
		&=\frac{1}{n^2}\sum_{i\ne j}\left(t_D(\pi(i))t_H(i)-t_D(\pi(i))t_H(j)-t_D(\pi(j))t_H(i)+t_D(\pi(j))t_H(j)\right)\\&=\frac{1}{n^2}\left(2(n-1)\sum_{i}t_D(\pi(i))t_H(i)-\sum_{i\ne j}t_D(\pi(i))t_H(j)-\sum_{i\ne j}t_D(\pi(j))t_H(i)\right)\\&=\frac{1}{n^2}\left(2(n-1)\sum_{i}t_D(\pi(i))t_H(i)-2\sum_it_D(\pi(i))\sum_{j\ne i}t_H(j)\right)\\&=\frac{1}{n^2}\left(2(n-1)\sum_{i}t_D(\pi(i))t_H(i)+2\sum_it_D(\pi(i))t_H(i)\right)\\&=\frac{2}{n}\sum_i t_D(\pi(i))t_H(i).
	\end{align*}
	Finally, applying the same permutation to the rows and columns preserves symmetry and skew-symmetry. Hence $(S_D^{(1)})_\pi$ is symmetric and $(K_D^{(1)})_\pi$ is
	skew-symmetric, so $\left\langle (S_D^{(1)})_\pi,K_H^{(1)}\right\rangle=\left\langle (K_D^{(1)})_\pi,S_H^{(1)}\right\rangle=0$. Consequently,
	\[
	\left\langle
	\bigl(S_D^{(1)}+K_D^{(1)}\bigr)_\pi,
	S_H^{(1)}+K_H^{(1)}
	\right\rangle=
	\frac{2}{n-2}\sum_i r_D(\pi(i))r_H(i)+
	\frac{2}{n}\sum_i t_D(\pi(i))t_H(i),
	\]
	as required.
\end{proof}
\subsection{Four-vertex support}\label{subs:four-vertex}
For distinct vertices $i,j$, denote
\begin{equation*}
({q_D})_{ij}=
\begin{cases}
	1, & \text{if } i\to j\in E(D) \text{ or } j\to i\in E(D),\\
	0, & \text{otherwise}.
\end{cases}
\end{equation*}
By definition, $q_D$ is symmetric.

For a four-set $S=\{a,b,c,d\}$, consider its three perfect matchings and let $m_1(S)=(q_D)_{ab}+(q_D)_{cd}$, $m_2(S)=(q_D)_{ac}+(q_D)_{bd}$ and $m_3(S)=(q_D)_{ad}+(q_D)_{bc}$. Thus $m_1(S),m_2(S),m_3(S)$ are the numbers of support edges in the three
perfect matchings of $S$. Define the support of $S$ by
\begin{equation*}
Q(D)=\sum_{S\in\binom{V(D)}{4}}Q_S(D)=\sum_{S\in\binom{V(D)}{4}}\sum_{1\leq i<j\leq3}\left((m_i(S)-m_j(S)\right)^2.
\end{equation*}

\begin{lemma}\label{thm:four-point-identity}
For $n\ge4$, we have $Q(D)=2(n-1)(n-2)\|R_D\|_{\mathrm F}^2$, where $\|R_D\|_{\mathrm F}^2=\sum_{i\neq j} {(R_D)}_{ij}^2$.
\end{lemma}
\begin{proof}[\bf Proof]
For $i\ne j$, by (\ref{eq:SDKD}), we have
\begin{equation*}
({q_D})_{ij}=p+2(S_D)_{ij}= p + 2(S_D^{(1)})_{ij} + 2(R_D)_{ij}.
\end{equation*}
Thus, for distinct $a,b,c,d$, we have $ (r_D(a) + r_D(b)) + (r_D(c) + r_D(d)) - (r_D(a) + r_D(c)) - (r_D(b) + r_D(d))=0$. By (\ref{eq:SD1KD1}), we obtain
\begin{equation*}
m_1(S)-m_2(S)=({q_D})_{ab}+({q_D})_{cd}-({q_D})_{ac}-({q_D})_{bd}=2\left((R_D)_{ab}+(R_D)_{cd}-(R_D)_{ac}-(R_D)_{bd}\right),
\end{equation*}
and similarly for the other two differences. 
\begin{equation*}
m_1(S)-m_3(S)=2\left((R_D)_{ab}+(R_D)_{cd}-(R_D)_{ad}-(R_D)_{bc}\right),
\end{equation*}
and
\begin{equation*}
m_2(S)-m_3(S)=2\left((R_D)_{ac}+(R_D)_{bd}-(R_D)_{ad}-(R_D)_{bc}\right).
\end{equation*}
 Let $r_1(S)=(R_D)_{ab}+(R_D)_{cd}$, $r_2(S)=(R_D)_{ac}+(R_D)_{bd}$ and $r_3(S)=(R_D)_{ad}+(R_D)_{bc}$. Hence we have
 \begin{equation*}
Q(D) = \sum_{S \in \binom{V(D)}{4}}Q_S(D)= 4 \sum_{S \in \binom{V(D)}{4}}\sum_{1\leq i<j\leq3}\left((r_i(S)-r_j(S)\right)^2.
 \end{equation*}
For a fixed four-set $S$, summing over all $24$ orderings $(a,b,c,d)$ of its vertices gives
\begin{equation*}
\sum_{\substack{(a,b,c,d)\\ \text{ are ordered}}}
\bigl((R_D)_{ab}+(R_D)_{cd}-(R_D)_{ac}-(R_D)_{bd}\bigr)^2=8\sum_{1\leq i<j\leq3}\left((r_i(S)-r_j(S)\right)^2=2Q_S(D).
\end{equation*}
Therefore
\begin{equation*}
2Q(D)=\sum_{\substack{(a,b,c,d)\\ \text{are distinct}}}\bigl((R_D)_{ab}+(R_D)_{cd}-(R_D)_{ac}-(R_D)_{bd}\bigr)^2.
\end{equation*}
Since $R_D$ is a symmetric matrix, we have $(R_D)_{bd}=(R_D)_{db}$ and $(R_D)_{cd}=(R_D)_{dc}$. Let $x_a=(R_D)_{ab} - (R_D)_{ac}$ and $x_d=(R_D)_{db}-(R_D)_{dc}$. We have
\begin{equation}\label{eq:xaxd}
\sum_{\substack{a, d \notin \{b, c\} \\ a \neq d}}x_a  x_d=\sum_{\substack{a, d \notin \{b, c\}}} x_a x_d - \sum_{a \notin \{b, c\}} x_a^2= \left( \sum_{a \notin \{b, c\}} x_a \right)^2 - \sum_{a \notin \{b, c\}} x_a^2,
\end{equation}
and by (\ref{eq:RD}), we have
\begin{equation}\label{eq:sumxa}
\sum_{a \notin \{b,c\}} x_a=\sum_{a \notin \{b,c\}}\bigl((R_D)_{ab} - (R_D)_{ac}\bigr)=-(R_D)_{cb}+(R_D)_{bc}=0.
\end{equation}
By (\ref{eq:xaxd}) and (\ref{eq:sumxa}), we obtain
\begin{align*}
2Q(D)&=\sum_{\substack{(a,b,c,d)\\ \text{are distinct}}}\bigl((R_D)_{ab}+(R_D)_{cd}-(R_D)_{ac}-(R_D)_{bd}\bigr)^2\\&=\sum_{b \neq c} \sum_{\substack{a, d \notin \{b, c\} \\ a \neq d}} \left(\bigl((R_D)_{ab} - (R_D)_{ac}\bigr) - \bigl((R_D)_{db} - (R_D)_{dc}\bigr) \right)^2\\&=\sum_{b \neq c}\left( 2(n-3)\sum_{\substack{a \notin \{b, c\} }}x_a^2-2\sum_{\substack{a, d \notin \{b, c\} \\ a \neq d}}x_ax_d \right)\\&=\sum_{b \neq c}\left( 2(n-3)\sum_{\substack{a \notin \{b, c\} }}x_a^2-2\left(\left( \sum_{a \notin \{b, c\}} x_a \right)^2 - \sum_{a \notin \{b, c\}} x_a^2\right) \right)\\&=2(n - 2) \sum_{b \neq c} \sum_{\substack{a \notin \{b, c\}}} x_a^2 - 2 \sum_{b \neq c} \left( \sum_{a \notin \{b,c\}} x_a \right)^2\\&=2(n - 2) \sum_{\substack{(a, b, c) \\ \text{all distinct}}} \bigl((R_D)_{ab} - (R_D)_{ac}\bigr)^2\\&=2(n - 2) \sum_a \left( 2(n - 1) \sum_{b \neq a} (R_D)_{ab}^2 - 2 \left( \sum_{b \neq a} (R_D)_{ab} \right)^2 \right)\\&=4(n - 1)(n - 2) \sum_{a} \sum_{b \neq a} (R_D)_{ab}^2\\&=4(n - 1)(n - 2) \|R_D\|_{\mathrm F}^2.
\end{align*}
This completes the proof.
\end{proof}
\subsection{Triangle circulation}
For distinct vertices $i,j$, define
\begin{equation*}
(c_D)_{ij}=
\begin{cases}
	1, & i\to j\in E(D),\\
	-1, & j\to i\in E(D),\\
	0, & \text{otherwise}.
\end{cases}
\end{equation*}
For each three-element set \(\{a, b, c\}\), choose either one of the two cyclic orientations of \(\{a, b, c\}\). Reversing this order changes the sign of \((c_D)_{ab} + (c_D)_{bc} + (c_D)_{ca}\), so its square is independent of the choice. Define
\begin{equation*}
C(D)=\sum_{\{a,b,c\}\in\binom{V(D)}{3}}\bigl((c_D)_{ab}+(c_D)_{bc}+(c_D)_{ca}\bigr)^2.
\end{equation*}
\begin{lemma}\label{thm:curl-identity}
For $n\ge3$, we have $C(D)=2n \|C_D\|^2_{\mathrm F}$, where $\|C_D\|^2_{\mathrm F}=\sum_{i\neq j}(C_D)_{ij}^2$.
\end{lemma}
\begin{proof}[\bf Proof]
	For $i\neq j$, by (\ref{eq:SDKD}), we have
	\begin{equation*}
	(c_D)_{ij}=2(K_D)_{ij}=2(K_D^{(1)})_{ij}+2(C_D)_{ij}.
	\end{equation*}
For a three-set $\{a,b,c\}$, by (\ref{eq:SD1KD1}), we have
\begin{equation*}
	(K_D^{(1)})_{ab}+(K_D^{(1)})_{bc}+(K_D^{(1)})_{ca}
	=\frac1n\left((t_D(a)-t_D(b))+(t_D(b)-t_D(c))+(t_D(c)-t_D(a))\right)=0.
\end{equation*}
Thus, we obtain $(c_D)_{ab}+(c_D)_{bc}+(c_D)_{ca}=2((C_D)_{ab}+(C_D)_{bc}+(C_D)_{ca})$. Each three-element set has six orderings, all giving the same squared sum. Summing over all ordered triples of distinct vertices therefore gives
\begin{equation*}
\sum_{\substack{(a,b,c)\\\text{all distinct}}}
\bigl((C_D)_{ab}+(C_D)_{bc}+(C_D)_{ca}\bigr)^2=6\sum_{\{a,b,c\}\in\binom{V(D)}3}
\bigl((C_D)_{ab}+(C_D)_{bc}+(C_D)_{ca}\bigr)^2.
\end{equation*}
Since $C_D$ is skew-symmetric, (\ref{eq:CD}) gives
	\begin{align*}
	&\qquad\sum_{\substack{(a,b,c)\\\text{all distinct}}}
	\bigl((C_D)_{ab}+(C_D)_{bc}+(C_D)_{ca}\bigr)^2\\&=
	3(n-2)\sum_{a\ne b}(C_D)_{ab}^2+6\sum_b
	\left(\sum_{\substack{a,c\ne b\\a\ne c}}(C_D)_{ab}(C_D)_{bc}\right)\\&=3(n-2)\|C_D\|_{\mathrm{F}}^2+	6\sum_b\left(\sum_{a,c \neq b} (C_D)_{ab} (C_D)_{bc} - \sum_{a \neq b} (C_D)_{ab} (C_D)_{ba}\right)\\&=3(n-2)\|C_D\|_{\mathrm{F}}^2+	6\sum_b\left(\left(\sum_{a\ne b}(C_D)_{ab}\right)\left(\sum_{c\ne b}(C_D)_{bc}\right)-
	\sum_{a\ne b}(C_D)_{ab}(C_D)_{ba}\right)\\&=3(n-2)\|C_D\|_{\mathrm{F}}^2+6\sum_b\sum_{a\ne b}(C_D)_{ab}^2=3n\|C_D\|_{\mathrm{F}}^2.
	\end{align*}
Then, we obtain 
\begin{equation*}
C(D)=4 \sum_{S \in \binom{V(D)}{3}} \bigl((C_D)_{ab} + (C_D)_{bc} + (C_D)_{ca}\bigr)^2=2n\|C_D\|_{\mathrm F}^2.
\end{equation*}
This completes the proof.
\end{proof}

\section{Transpositions and discrepancy}\label{sec:structure-energy}

Fix an oriented graph $D$. For an ordered pair $x\ne y$ and a vertex $z\notin\{x,y\}$, define 
\begin{equation} \label{eq:row-difference-vector}
 u^D_{xy}(z)=
 \begin{pmatrix}
 (R_D)_{xz}-(R_D)_{yz}\\[1mm]
(C_D)_{xz}-(C_D)_{yz}+\dfrac{2(C_D)_{xy}}{n-2}
 \end{pmatrix}.
\end{equation}
\begin{lemma}\label{lem:row-vector-bounds}
For every $x\ne y$,
\begin{equation*}
 \sum_{z\notin\{x,y\}}u^D_{xy}(z)=\mathbf{0}.
\end{equation*}
Moreover, if \( n \geq 4 \), then $\|u^D_{xy}(z)\|\le7$ for all pairwise distinct vertices \( x, y, z \).
\end{lemma}
\begin{proof}[\bf Proof]
Since $R_D$ is symmetric, (\ref{eq:RD}) gives
\begin{equation*}
 \sum_{z\notin\{x,y\}}\bigl((R_D)_{xz}-(R_D)_{yz}\bigr)=-(R_D)_{xy}+(R_D)_{yx}=0.
\end{equation*}
Since $C_D$ is skew-symmetric. By (\ref{eq:CD}), we have
\begin{equation*}
 \sum_{z\notin\{x,y\}}\left(\bigl((C_D)_{xz}-(C_D)_{yz}\bigr)+\dfrac{2(C_D)_{xy}}{n-2}\right)=-(C_D)_{xy}+(C_D)_{yx}+2(C_D)_{xy}=0.
\end{equation*}
For the uniform bound, when $i\neq j$ we have $\left|(S_D)_{ij}\right|=\left|(q_D)_{ij}-p\right|/2\leq 1/2$. For $n\geq 4$,
\begin{equation*}
\left|(S^{(1)}_D)_{ij}\right|=\left|\frac{r_D(i)+r_D(j)}{n-2}\right|=\left|\frac{\sum_{k\neq i}(S_D)_{ik}+\sum_{k\neq j}(S_D)_{jk}}{n-2}\right|\leq\frac{\frac{n-1}{2}+\frac{n-1}{2}}{n-2}\leq\frac{3}{2}.
\end{equation*}
Thus, we have $\left|(R_D)_{ij}\right|=\left|(S_D)_{ij}-(S_D^{(1)})_{ij}\right|\leq 2$. Similarly, we have $\left|(K_D)_{ij}\right|\leq1/2$ and 
\begin{equation*}
\left|(K_D^{(1)})_{ij}\right|=\left|\frac{t_D(i)-t_D(j)}{n}\right|=\left|\frac{\sum_{k\neq i}(K_D)_{ik}-\sum_{k\neq j}(K_D)_{jk}}{n}\right|\leq\frac{\frac{n-1}{2}+\frac{n-1}{2}}{n}<1.
\end{equation*}
Thus, we have $|(C_D)_{ij}|=\left|(K_D)_{ij}-(K_D^{(1)})_{ij}\right|\leq 3/2$. Hence
\begin{equation*}
\|u_{xy}^D(z)\|^2 = \left( (R_D)_{xz} - (R_D)_{yz} \right)^2+\left( (C_D)_{xz} - (C_D)_{yz} + \frac{2(C_D)_{xy}}{n-2} \right)^2\leq 49,
\end{equation*}
which proves $\|u^D_{xy}(z)\|\le7$.
\end{proof}
Choose independently a uniformly random ordered pair $(u,v)$ of vertices
of $D$ and a uniformly random ordered pair $(x,y)$ of vertices of $H$. Let $\phi:V(H)\setminus\{x,y\}\longrightarrow V(D)\setminus\{u,v\}$ be a uniformly random bijection. Extend $\phi$ to two bijections $\pi_0^\phi,\pi_1^\phi:V(H)\to V(D)$ by setting
\begin{equation}\label{eq:xuyv}
\pi_0^\phi(x)=u,\quad
\pi_0^\phi(y)=v,\qquad
\pi_1^\phi(x)=v,\quad
\pi_1^\phi(y)=u,
\end{equation}
 while both maps agree with $\phi$ on $V(H)\setminus\{x,y\}$  $(\text{i.e. }\pi_0^\phi(z) = \pi_1^\phi(z) = \phi(z))$. Put $\Delta_\phi=\langle A_{\pi_1^\phi},B\rangle-\langle A_{\pi_0^\phi},B\rangle$.
\begin{lemma}\label{lem:exact-two-dimensional-swap}
There is a constant $c_{uvxy}$ depending only on $u,v,x,y$, and not on $\phi$, such that
\begin{equation*}
\Delta_\phi=c_{uvxy}-2\sum_{z\in V(H)\setminus\{x,y\}}\left(u^H_{xy}(z)\right)^{\mathsf T} u^D_{uv}(\phi(z)).
\end{equation*}
\end{lemma}
\begin{proof}[\bf Proof]
By Lemma~\ref{lem:orthogonal-decomposition}, we have $A=S_D^{(1)}+K_D^{(1)}+R_D+C_D$ and $B=S_H^{(1)}+K_H^{(1)}+R_H+C_H$. Hence
\begin{align}
	\Delta_\phi&=\langle A_{\pi_1^\phi},B\rangle-\langle A_{\pi_0^\phi},B\rangle \\
	&=\left\langle\left(S_D^{(1)}+K_D^{(1)}+R_D+C_D\right)_{\pi_1^\phi}-
	\left(S_D^{(1)}+K_D^{(1)}+R_D+C_D\right)_{\pi_0^\phi},S_H^{(1)}+K_H^{(1)}+R_H+C_H
	\right\rangle. \notag
\end{align}
Using orthogonality between symmetric and skew-symmetric matrices and the zero row sums in (\ref{eq:RD}) and (\ref{eq:CD}), we obtain 
\begin{align}
\Delta_\phi&=
\left(\left\langle
\left(S_D^{(1)}+K_D^{(1)}\right)_{\pi_1^\phi},
S_H^{(1)}+K_H^{(1)}
\right\rangle-
\left\langle
\left(S_D^{(1)}+K_D^{(1)}\right)_{\pi_0^\phi},
S_H^{(1)}+K_H^{(1)}
\right\rangle\right) \notag\\&+
\left(\left\langle (R_D)_{\pi_1^\phi},R_H\right\rangle-
\left\langle (R_D)_{\pi_0^\phi},R_H\right\rangle\right)+
\left(\left\langle (C_D)_{\pi_1^\phi},C_H\right\rangle-
\left\langle (C_D)_{\pi_0^\phi},C_H\right\rangle\right).
\end{align}
Using~(\ref{eq:first-order-correlation}), the difference of the first-order terms equals
\begin{align*}
	\frac{2}{n-2}
	\sum_{z\in V(H)}
	\left(
	r_D(\pi_1^\phi(z))-r_D(\pi_0^\phi(z))
	\right)r_H(z)+
	\frac2n
	\sum_{z\in V(H)}
	\left(
	t_D(\pi_1^\phi(z))-t_D(\pi_0^\phi(z))
	\right)t_H(z).
\end{align*}
Using (\ref{eq:xuyv}) and the fact that \( \pi_0^\phi(z) = \pi_1^\phi(z) = \phi(z) \) for every \( z \in V(H) \setminus \{x, y\} \), we obtain 
\begin{align}\label{eq:Detla_phi2}
	\Delta_\phi=&
	-\frac{2}{n-2}
	\bigl(r_D(u)-r_D(v)\bigr)
	\bigl(r_H(x)-r_H(y)\bigr)-\frac2n\bigl(t_D(u)-t_D(v)\bigr)
	\bigl(t_H(x)-t_H(y)\bigr) \notag\\&+
	\left(\left\langle (R_D)_{\pi_1^\phi},R_H\right\rangle-
	\left\langle (R_D)_{\pi_0^\phi},R_H\right\rangle
	\right)+\left(\left\langle (C_D)_{\pi_1^\phi},C_H\right\rangle-
	\left\langle (C_D)_{\pi_0^\phi},C_H\right\rangle\right).
\end{align}
We next consider the symmetric residual term. Since $R_D$ is symmetric, the
pair $\{x,y\}$ makes no contribution to the change. For each \( z \in V(H) \setminus \{x, y\} \), put \( w = \phi(z) \). Summing the changes involving \(\{x, z\}\) and \(\{y, z\}\) over all such \(z\), we obtain
\begin{align}\label{eq:R-R}
&\qquad\left\langle (R_D)_{\pi_1^\phi},R_H\right\rangle-
\left\langle (R_D)_{\pi_0^\phi},R_H\right\rangle\notag\\&=2\sum_{z\in V(H)\setminus\{x,y\}} \bigl( (R_D)_{vw}(R_H)_{xz} + (R_D)_{uw}(R_H)_{yz} - (R_D)_{uw}(R_H)_{xz} - (R_D)_{vw}(R_H)_{yz} \bigr)\notag\\&=-2\sum_{z\in V(H)\setminus\{x,y\}}((R_H)_{xz} - (R_H)_{yz})((R_D)_{uw} - (R_D)_{vw})\notag\\&=-2 \sum_{z \in V(H) \setminus \{x, y\}} u_{xy,1}^H(z) u_{uv,1}^D(\phi(z)).
\end{align}
For the skew-symmetric residual term, the contribution of the pair $\{x,y\}$ changes by $-4(C_D)_{uv}(C_H)_{xy}$. For each \( z \in V(H) \setminus \{x, y\} \), put \( w = \phi(z) \). The corresponding contribution associated with $z$ changes by
$-2\bigl((C_H)_{xz}-(C_H)_{yz}\bigr)\bigl((C_D)_{uw}-(C_D)_{vw}\bigr)$. Hence
\begin{align}\label{eq:C-C}
&\qquad\left\langle (C_D)_{\pi_1^\phi},C_H\right\rangle-
\left\langle (C_D)_{\pi_0^\phi},C_H\right\rangle\notag
\\&=
-4(C_D)_{uv}(C_H)_{xy}
-2\sum_{z\in V(H)\setminus\{x,y\}}
\bigl((C_H)_{xz}-(C_H)_{yz}\bigr)\bigl((C_D)_{u\phi(z)}-(C_D)_{v\phi(z)}\bigr).
\end{align}
Since $C_H$ and $C_D$ are skew-symmetric and have zero row sums by (\ref{eq:CD}), we obtain
\begin{equation*}
\sum_{z\in V(H)\setminus\{x,y\}}\bigl((C_H)_{xz}-(C_H)_{yz}\bigr)=-2(C_H)_{xy},\quad \sum_{w\in V(D)\setminus\{u,v\}}\bigl((C_D)_{uw}-(C_D)_{vw}\bigr)=-2(C_D)_{uv}.
\end{equation*}
Then
\begin{align}\label{eq:CD-CD}
&\qquad	\sum_{z\in V(H)\setminus\{x,y\}}
	\bigl((C_H)_{xz}-(C_H)_{yz}\bigr)\bigl((C_D)_{u\phi(z)}-(C_D)_{v\phi(z)}\bigr)\notag\\&=
	\sum_{z\in V(H)\setminus\{x,y\}}
	u^H_{xy,2}(z)u^D_{uv,2}(\phi(z))+\frac{4(C_H)_{xy}(C_D)_{uv}}{n-2}.
\end{align}
Substituting \eqref{eq:CD-CD} into \eqref{eq:C-C}, we obtain
\begin{equation}\label{eq:CD1-CD1}
\langle (C_D)_{\pi_1^\phi},C_H\rangle
-\langle (C_D)_{\pi_0^\phi},C_H\rangle=
-2\sum_{z\in V(H)\setminus\{x,y\}}
u^H_{xy,2}(z)u^D_{uv,2}(\phi(z))-
\frac{4n}{n-2}(C_D)_{uv}(C_H)_{xy}.
\end{equation}
Define
\begin{align*}
c_{uvxy}=&-\frac{2}{n-2}
\bigl(r_D(u)-r_D(v)\bigr)
\bigl(r_H(x)-r_H(y)\bigr)-
\frac2n
\bigl(t_D(u)-t_D(v)\bigr)
\bigl(t_H(x)-t_H(y)\bigr)\\&-
\frac{4n}{n-2}(C_D)_{uv}(C_H)_{xy}.
\end{align*}
Combining \eqref{eq:Detla_phi2}, \eqref{eq:R-R} and \eqref{eq:CD1-CD1}, we obtain
\begin{align*}
	\Delta_\phi
&=c_{uvxy}-2\sum_{z\in V(H)\setminus\{x,y\}}
\left(u^H_{xy,1}(z)u^D_{uv,1}(\phi(z))+u^H_{xy,2}(z)u^D_{uv,2}(\phi(z))
\right)\\&=c_{uvxy}-2\sum_{z\in V(H)\setminus\{x,y\}}
\bigl(u^H_{xy}(z)\bigr)^{\mathsf T}
u^D_{uv}(\phi(z)),
\end{align*}
as required.
\end{proof}
Define 
\begin{equation*}
 \Sigma_D(x,y)=\sum_{z\notin\{x,y\}}u^D_{xy}(z)\left(u^D_{xy}(z)\right)^{\mathsf T}.
\end{equation*}
\begin{lemma}\label{lem:average-row-covariance}
Let $n\ge 4$, and let $(x,y)$ be chosen uniformly at random from the
ordered pairs of distinct vertices of $D$. Then
\begin{equation*}
 \E_{x\ne y}\Sigma_D(x,y)=
\begin{pmatrix}
	\dfrac{2\|R_D\|_{\mathrm F}^2}{n}&0\\[2mm]
	0&\dfrac{2(n-3)\|C_D\|_{\mathrm F}^2}{(n-1)(n-2)}
\end{pmatrix}.
\end{equation*}
\end{lemma}
\begin{proof}[\bf Proof]
 For the $\bigl(\E_{x\ne y}\Sigma_D(x,y)\bigr)_{11}$, (\ref{eq:RD}) gives, 
 \begin{align}\label{eq:E11}
 &\qquad\mathbb E_{x\ne y}
 \sum_{z\notin\{x,y\}}
 \bigl((R_D)_{xz}-(R_D)_{yz}\bigr)^2\notag\\
 &=
 \frac{1}{n(n-1)}
 \sum_z
 \sum_{\substack{x\ne y\\x,y\ne z}}
 \bigl((R_D)_{xz}-(R_D)_{yz}\bigr)^2\notag\\
 &=
 \frac{1}{n(n-1)}
 \sum_z
 \left(
 2(n-2)\sum_{x\ne z}(R_D)_{xz}^2-
 2\sum_{\substack{x\ne y\\x,y\ne z}}
 (R_D)_{xz}(R_D)_{yz}
 \right)\notag\\
 &=
 \frac{1}{n(n-1)}
 \sum_z
 \left(
 2(n-2)\sum_{x\ne z}(R_D)_{xz}^2-
 2\left(\left(\sum_{x\ne z}(R_D)_{xz}\right)^2-\sum_{x\ne z}(R_D)_{xz}^2\right)\right)\notag\\&=
 \frac{2(n-1)}{n(n-1)}
 \sum_z\sum_{x\ne z}(R_D)_{xz}^2=\frac{2\|R_D\|_{\mathrm{F}}^2}{n}.
 \end{align}
	For the $\bigl(\E_{x\ne y}\Sigma_D(x,y)\bigr)_{22}$, the same calculation gives
	\begin{equation*}
		\sum_z\sum_{\substack{x\ne y\\x,y\ne z}}\bigl((C_D)_{xz}-(C_D)_{yz}\bigr)^2=
	2(n-1)\|C_D\|_{\mathrm{F}}^2.
	\end{equation*}
	Moreover, for fixed $x\ne y$, by (\ref{eq:CD}) and skew-symmetry of $C_D$,
	\[
	\sum_{z\notin\{x,y\}}
	\bigl((C_D)_{xz}-(C_D)_{yz}\bigr)
	=
	-(C_D)_{xy}+(C_D)_{yx}
	=
	-2(C_D)_{xy}.
	\]
	Hence
	\[
	\begin{aligned}
		&\qquad\sum_{z\notin\{x,y\}}
		\left(
		(C_D)_{xz}-(C_D)_{yz}
		+\frac{2(C_D)_{xy}}{n-2}
		\right)^2\\
		&=
		\sum_{z\notin\{x,y\}}
		\bigl((C_D)_{xz}-(C_D)_{yz}\bigr)^2+
		\frac{4(C_D)_{xy}}{n-2}
		\sum_{z\notin\{x,y\}}
		\bigl((C_D)_{xz}-(C_D)_{yz}\bigr)
		+
		\frac{4(C_D)_{xy}^2}{n-2}\\
		&=
		\sum_{z\notin\{x,y\}}
		\bigl((C_D)_{xz}-(C_D)_{yz}\bigr)^2
		-
		\frac{4(C_D)_{xy}^2}{n-2}.
	\end{aligned}
	\]
	Therefore
	\begin{align}\label{eq:E22}
	&\qquad\mathbb E_{x\ne y}
	\sum_{z\notin\{x,y\}}
	\left(
	(C_D)_{xz}-(C_D)_{yz}
	+\frac{2(C_D)_{xy}}{n-2}
	\right)^2\notag\\&=
	\frac{1}{n(n-1)}
	\left(
	2(n-1)\|C_D\|_{\mathrm{F}}^2-\frac{4\|C_D\|_{\mathrm{F}}^2}{n-2}
	\right)=
	\frac{2(n-3)\|C_D\|_{\mathrm{F}}^2}{(n-1)(n-2)}.
	\end{align}
	Since $ \Sigma_D(x,y)$ is symmetric, we have $\bigl(\E_{x\ne y}\Sigma_D(x,y)\bigr)_{12}=\bigl(\E_{x\ne y}\Sigma_D(x,y)\bigr)_{21}$. By (\ref{eq:RD}), we obtain
	\begin{equation}\label{eq:E12}
		\sum_{z\notin\{x,y\}}\left(\bigl((R_D)_{xz}-(R_D)_{yz}\bigr)\cdot\frac{2(C_D)_{xy}}{n-2}\right)=0.
	\end{equation}
Using (\ref{eq:RD}), (\ref{eq:CD}), and the symmetry of \( R_D \) and skew-symmetry of \( C_D \), we obtain
	\begin{align}\label{eq:E21}
		&\qquad\sum_z
	\sum_{\substack{x\ne y\\x,y\ne z}}
	\bigl((R_D)_{xz}-(R_D)_{yz}\bigr)
	\bigl((C_D)_{xz}-(C_D)_{yz}\bigr)\notag\\&=
	\sum_z\left(2(n-2)\sum_{x\ne z}(R_D)_{xz}(C_D)_{xz}-
	2\sum_{\substack{x\ne y\\x,y\ne z}}(R_D)_{xz}(C_D)_{yz}\right)\notag\\&=
	\sum_z\left(2(n-2)\sum_{x\ne z}(R_D)_{xz}(C_D)_{xz}-2\left(
	\left(\sum_{x\ne z}(R_D)_{xz}\right)
	\left(\sum_{y\ne z}(C_D)_{yz}\right)-
	\sum_{x\ne z}(R_D)_{xz}(C_D)_{xz}\right)\right)\notag\\&=
	2(n-1)
	\sum_z\sum_{x\ne z}
	(R_D)_{xz}(C_D)_{xz}\notag\\&=
	2(n-1)\langle R_D,C_D\rangle=0.
	\end{align}
	Combining (\ref{eq:E11}), (\ref{eq:E22}), (\ref{eq:E12}) and (\ref{eq:E21}) proves the lemma.
\end{proof}

\begin{lemma}\label{lem:local-gamma}
There is an absolute constant $c_3>0$ such that, for every pair of
oriented graphs $D,H$ of order $n\ge4$,
\begin{equation*}
\gamma(A,B)\ge c_3\frac{\left(\|R_D\|_{\mathrm F}^2\|R_H\|_{\mathrm F}^2+\|C_D\|_{\mathrm F}^2\|C_H\|_{\mathrm F}^2\right)^{3/2}}{n^{11/2}}.
\end{equation*}
\end{lemma}
\begin{proof}[\bf Proof]
	Fix an ordered pair \( (u, v) \) of distinct vertices of \( D \) and an ordered pair \( (x, y) \) of distinct vertices of \( H \). By Lemma~\ref{lem:exact-two-dimensional-swap}, if
	\begin{equation*}
		Y_\phi=\sum_{z\in V(H)\setminus\{x,y\}}\bigl(u^H_{xy}(z)\bigr)^{\mathsf T}u^D_{uv}(\phi(z)),
	\end{equation*}
	then $\Delta_\phi=c_{uvxy}-2Y_\phi$. By Lemmas~\ref{lem:EYphi} and \ref{lem:row-vector-bounds}, we have
	\begin{equation*}
		\mathbb E_\phi|Y_\phi|\ge c\, \frac{\left(\operatorname{tr}\bigl(\Sigma_D(u,v)\Sigma_H(x,y)\bigr)\right)^{3/2}}{n^{5/2}},
	\end{equation*}
	for an absolute constant $c>0$. Moreover,
	\begin{align*}
	\mathbb E_\phi Y_\phi
	&=
	\sum_{z\in V(H)\setminus\{x,y\}}
	\bigl(u^H_{xy}(z)\bigr)^{\mathsf T}
	\mathbb E_\phi u^D_{uv}(\phi(z))\\
	&=
	\frac{1}{n-2}
	\sum_{z\in V(H)\setminus\{x,y\}}
	\bigl(u^H_{xy}(z)\bigr)^{\mathsf T}
	\sum_{w\in V(D)\setminus\{u,v\}}u^D_{uv}(w)
	=0.
	\end{align*}
	Hence Proposition~\ref{prop:constant-shift} applied to $-2Y_\phi$ yields $\mathbb E_\phi|\Delta_\phi|=
	\mathbb E_\phi|c_{uvxy}-2Y_\phi|\ge
	\mathbb E_\phi|2Y_\phi|/2=\mathbb E_\phi|Y_\phi|$. Therefore
	\begin{equation}\label{eq:Ephi}
		\mathbb E_\phi|\Delta_\phi|\ge c\, \frac{\operatorname{tr}\bigl(\Sigma_D(u,v)\Sigma_H(x,y)\bigr)^{3/2}}{n^{5/2}}.
	\end{equation}
		The random relabelling in the definition of $\gamma(A,B)$
	can be generated as follows: choose independently and uniformly an
	ordered pair $(u,v)$ for $A$ and an ordered pair $(x,y)$ for $B$, and
	then choose a uniformly random bijection $\phi:[n]\setminus\{x,y\}\longrightarrow[n]\setminus\{u,v\}$. The fixed transposition interchanges the assignments $	x\mapsto u$, $y\mapsto v$ with $x\mapsto v$, $y\mapsto u$ while leaving $\phi$ unchanged on all remaining
	vertices. Hence the resulting change in correlation is precisely
	$\Delta_\phi$, and therefore
	\begin{equation*}
		\gamma(A,B)=\E_{\rho,\sigma}\left|\langle (A_{\rho})_\tau,B_\sigma\rangle-\langle A_\rho,B_\sigma\rangle\right|=\E_{(u,v),(x,y)}\E_\phi|\Delta_\phi|.
	\end{equation*}
	By \eqref{eq:Ephi} and Jensen's inequality, we have
	\begin{equation}\label{eq:rAB}
	\gamma(A,B)\ge\frac{c}{n^{5/2}}\left(\mathbb E_{(u,v),(x,y)}\operatorname{tr}\bigl(\Sigma_D(u,v)\Sigma_H(x,y)\bigr)\right)^{3/2}.
	\end{equation}
	The two ordered pairs are independent, and hence we have $\mathbb E_{(u,v),(x,y)}\operatorname{tr}\bigl(\Sigma_D(u,v)\Sigma_H(x,y)\bigr)=\operatorname{tr}
	\left(\left(\mathbb E_{u\ne v}\Sigma_D(u,v)\right)\left(\mathbb E_{x\ne y}\Sigma_H(x,y)\right)\right)$.
	By Lemma~\ref{lem:average-row-covariance}, 
	\begin{equation}\label{eq:Euvxy}
	\mathbb E_{(u,v),(x,y)}\operatorname{tr}\bigl(\Sigma_D(u,v)\Sigma_H(x,y)\bigr)
	=\frac{4\|R_D\|_{\mathrm F}^2\|R_H\|_{\mathrm F}^2}{n^2}+
	\frac{4(n-3)^2\|C_D\|_{\mathrm F}^2\|C_H\|_{\mathrm F}^2}{(n-1)^2(n-2)^2}.
	\end{equation}
	For $n\ge4$, substituting \eqref{eq:Euvxy} into \eqref{eq:rAB} gives
	\begin{equation*}
	\gamma(A,B)\ge c_3\frac{\left(\|R_D\|_{\mathrm F}^2\|R_H\|_{\mathrm F}^2+
		\|C_D\|_{\mathrm F}^2\|C_H\|_{\mathrm F}^2\right)^{3/2}}{n^{11/2}},
	\end{equation*}
	as required.
\end{proof}
\begin{theorem}\label{thm:structural-energy}
There is an absolute constant $c>0$ such that, for every pair of
oriented graphs $D,H$ of order $n\ge4$,
\begin{equation*}
\operatorname{disc}^{+}(D,H)\operatorname{disc}^{-}(D,H)\ge c\,
\frac{\left(\|R_D\|_{\mathrm F}^{2}\|R_H\|_{\mathrm F}^{2}+\|C_D\|_{\mathrm F}^{2}\|C_H\|_{\mathrm F}^{2}\right)^3}{n^9}.
\end{equation*}
\end{theorem}
\begin{proof}[\bf Proof]
The result follows by combining Lemmas~\ref{lem:amplification} and \ref{lem:local-gamma}.
\end{proof}

\section{Proofs of the main results}\label{sec:mainresults}
For an oriented graph $D$, let $G_D$ denote its underlying support graph;
that is, $xy\in E(G_D)$ if and only if one of $x\to y$ and $y\to x$
belongs to $E(D)$.
\subsection{Tournaments}
\begin{proof}[\bf Proof of Theorem \ref{thm:tournament-intro}]
Assume first that $n\ge 4$. For a tournament $T$, we have $(c_T)_{xy}\in\{-1,1\}$ for every pair of distinct vertices $x,y$. Hence, for every three-set $\{a,b,c\}$, $(c_T)_{ab}+(c_T)_{bc}+(c_T)_{ca} \in\{-3,-1,1,3\}$. Therefore
\begin{equation*}
C(T)=\sum_{\{a,b,c\}\in\binom{V(T)}{3}}\bigl((c_T)_{ab}+(c_T)_{bc}+(c_T)_{ca}\bigr)^2\ge\binom{n}{3}.
\end{equation*}
By Lemma~\ref{thm:curl-identity}, we have $\|C_T\|_{\mathrm F}^2={C(T)}/{2n}\ge
{(n-1)(n-2)}/{12}$. The same estimate holds for $U$. Hence, by Theorem~\ref{thm:structural-energy},
\begin{equation*}
\operatorname{disc}^{+}(T,U)\operatorname{disc}^{-}(T,U)\ge c\,
\frac{\left(\|R_T\|_{\mathrm F}^{2}\|R_U\|_{\mathrm F}^{2}+\|C_T\|_{\mathrm F}^{2}\|C_U\|_{\mathrm F}^{2}\right)^3}{n^9}\ge c\,\frac{\left(\|C_T\|_{\mathrm F}^{2}\|C_U\|_{\mathrm F}^{2}\right)^3}{n^9}\ge c_1 n^3,
\end{equation*}
for an absolute constant $c_1>0$. For $n=2,3$, a direct check gives $\operatorname{disc}^{+}(T,U)\ge1/2$ and $\operatorname{disc}^{-}(T,U)\ge1/2$. After decreasing the absolute constant if necessary, the product bound therefore holds for every \( n \geq 2 \). The discrepancy bound follows from \( \text{disc}(T, U)^2 \geq \text{disc}^+(T, U) \, \text{disc}^-(T, U) \). This completes the proof.
\end{proof}
\subsection{Good four-cycles}

Given a graph $G$ with vertex set $V$, and a $4$-cycle $C$ in the complete
graph $K_V$, Bollob\'as and Scott \cite{BollobasScottGraphs2011} say that $C$ is \emph{good} (in $G$) if either $|E(G)\cap E(C)|$ is odd, or $E(G)\cap E(C)$ is a pair of vertex-disjoint
edges. (So the bad cases are when $E(C)\subset E(G)$ or $E(C)\cap E(G)=\varnothing$ or $E(G)\cap E(C)$ is a path of length $2$.)
\begin{figure}[ht]
	\centering
	\begin{tikzpicture}[
		scale=1.15,
		vertex/.style={circle,fill=black,inner sep=2.8pt}]
		\begin{scope}[xshift=0cm]
			\node[vertex] (a1) at (0,1.6) {};
			\node[vertex] (b1) at (2.8,1.6) {};
			\node[vertex] (c1) at (0,0) {};
			\node[vertex] (d1) at (2.8,0) {};
			\draw (a1)--(b1);
			\draw (a1)--(d1);
			\draw (c1)--(b1);
			\draw[dashed] (c1)--(d1);
		\end{scope}
		\begin{scope}[xshift=4.8cm]
			\node[vertex] (a2) at (0,1.6) {};
			\node[vertex] (b2) at (2.8,1.6) {};
			\node[vertex] (c2) at (0,0) {};
			\node[vertex] (d2) at (2.8,0) {};
			\draw[dashed] (a2)--(b2);
			\draw (a2)--(d2);
			\draw (c2)--(b2);
			\draw[dashed] (c2)--(d2);
		\end{scope}
		\begin{scope}[xshift=9.6cm]
			\node[vertex] (a3) at (0,1.6) {};
			\node[vertex] (b3) at (2.8,1.6) {};
			\node[vertex] (c3) at (0,0) {};
			\node[vertex] (d3) at (2.8,0) {};
			\draw (a3)--(b3);
			\draw[dashed] (a3)--(d3);
			\draw[dashed] (c3)--(b3);
			\draw[dashed] (c3)--(d3);
		\end{scope}
		\node at (6.2,-0.6) {\itshape Good $4$-cycles};
	\end{tikzpicture}
\end{figure}
\begin{lemma}[\textbf{Bollob\'as and Scott} \cite{BollobasScottGraphs2011}]\label{lem:jgt2011}
Suppose that \( G \) has \( n \) vertices and \( p \binom{n}{2} \) edges, where \( \min\{p, 1-p\} \geq 16/n \). Then the probability that a 4-cycle chosen uniformly at random is good is at least $p^2(1-p)^2 / 5040$.
\end{lemma}
\begin{lemma}\label{lem:QD>gD}
	Let $G$ be a graph, let $D$ be any orientation of $G$, and let \( g(G) \) denote the number of 4-cycles in \( K_{V(G)} \) that are good for \( G \). Then
	\begin{equation*}
	Q(D)\ge\frac23 g(G).
	\end{equation*}
\end{lemma}

\begin{proof}[\bf Proof]
	Fix a four-set $S$, and let $m_1(S),m_2(S),m_3(S)$ be the numbers of
	support edges in its three perfect matchings, as in Section~\ref{subs:four-vertex}.
	\begin{claim}\label{claim:m1m2m3}
	If there is at least one good $4$-cycle on $S$, then $m_1(S)$, $m_2(S)$, $m_3(S)$ are not all equal.
	\end{claim}
	\begin{proof}
	Suppose, to the contrary, that $m_1(S)=m_2(S)=m_3(S)=k$. Since each perfect matching contains two edges, we have $k\in\{0,1,2\}$.
	
	\item \textbf{Case 1. $k=0$.}
	
	The three perfect matchings of $S$ partition the six edges of $K_4$.
	Thus $m_1(S)=m_2(S)=m_3(S)=0$ implies $E(G[S])=\varnothing$. Hence every $4$-cycle on $S$ contains no edge of $G$, and therefore is not good.
	
	\item \textbf{Case 2. $k=2$.}
	
	In this case every edge in each of the three perfect matchings belongs
	to $G[S]$. Hence $G[S]=K_4$. Thus every $4$-cycle on $S$ contains all four of its edges in $G$, and hence no such cycle is good.
	
	\item \textbf{Case 3. $k=1$.}
	
	Since the three perfect matchings partition the edges of $K_4$, $e(G[S])=m_1(S)+m_2(S)+m_3(S)=3$. Moreover, $G[S]$ cannot contain two vertex-disjoint edges. Indeed, such two edges form one of the three perfect matchings, which would
	force the corresponding $m_i(S)$ to be $2$, a contradiction. Hence the three edges of $G[S]$ are pairwise intersecting, and therefore $G[S]\cong K_{1,3}$ or $G[S]\cong K_3\cup K_1$.
	
	If $G[S]\cong K_{1,3}$, every $4$-cycle on $S$ contains exactly two
	edges of $G$, and these two edges are adjacent. If
	$G[S]\cong K_3\cup K_1$, every $4$-cycle on $S$ again contains exactly
	two edges of $G$, and these two edges are adjacent. Thus no $4$-cycle
	on $S$ is good.
	
	In all three cases, $S$ contains no good $4$-cycle, contradicting the
	assumption. Therefore $m_1(S),m_2(S),m_3(S)$ are not all equal.
	\end{proof}
	By Claim \ref{claim:m1m2m3}, if $S$ contains a good $4$-cycle, then
	\begin{equation*}
		\bigl(m_1(S)-m_2(S)\bigr)^2+\bigl(m_1(S)-m_3(S)\bigr)^2+\bigl(m_2(S)-m_3(S)\bigr)^2\ge2.
	\end{equation*}
	On the other hand, there are exactly three $4$-cycles on a fixed
	four-set $S$. Thus, if $S$ contains at least one good $4$-cycle, its
	contribution to $Q(D)$ is at least $2$, while the number of good
	$4$-cycles on $S$ is at most $3$. Hence 
	\begin{equation*}
	Q(D)= \sum_{S \in \binom{V(G)}{4}} Q_S(D)\geq\frac{2}{3} \sum_{S \in \binom{V(G)}{4}} g(S)= \frac{2}{3}g(G),
	\end{equation*}
	as required.
\end{proof}
\subsection{Oriented graphs of moderate density}
\begin{proof}[\bf Proof of Theorem~\ref{thm:moderate-density}]
	Let $G_D$ and $G_H$ be the underlying support graphs of $D$ and $H$,
	respectively. Since $e(G_D)=e(D)=p\binom{n}{2}$ and $e(G_H)=e(H)=q\binom{n}{2}$, the assumptions of Theorem~\ref{thm:moderate-density} imply $\min\{p,1-p\}\ge{16}/{n}$ and $\min\{q,1-q\}\ge{16}/{n}$. There are $3\binom{n}{4}$ distinct $4$-cycles in $K_n$. By Lemmas~\ref{lem:jgt2011} and \ref{lem:QD>gD}, we have
	\begin{equation*}
		Q(D)\ge\frac23g(G_D)\ge\frac{p^2(1-p)^2}{2520}\binom{n}{4}.
	\end{equation*}
	Using Lemma~\ref{thm:four-point-identity}, we have
	\begin{equation*}
		\|R_D\|_{\mathrm F}^2=
	\frac{Q(D)}{2(n-1)(n-2)}\ge
	\frac{p^2(1-p)^2}{5040}
	\frac{\binom{n}{4}}{(n-1)(n-2)}=
	\frac{p^2(1-p)^2}{120960}\,n(n-3).
	\end{equation*}
	The hypotheses imply $n\ge32$, and hence $n-3\ge n/2$. Therefore $\|R_D\|_{\mathrm F}^2
	\ge c_f p^2(1-p)^2n^2$, where $c_f>0$ is an absolute constant. Applying the same argument to $H$ gives $	\|R_H\|_{\mathrm F}^2\ge c_f q^2(1-q)^2n^2$. Consequently, $\|R_D\|_{\mathrm F}^2\|R_H\|_{\mathrm F}^2\ge c_f^2\bigl(p(1-p)q(1-q)\bigr)^2n^4$. By Theorem~\ref{thm:structural-energy}, we have
	\begin{equation*}
	\operatorname{disc}^{+}(D,H)\operatorname{disc}^{-}(D,H)\ge  c\,
	\frac{\left(\|R_D\|_{\mathrm F}^{2}\|R_H\|_{\mathrm F}^{2}\right)^3}{n^9}\ge c_2
	\bigl(p(1-p)q(1-q)\bigr)^6n^3,
	\end{equation*}
	for an absolute constant $c_2>0$. Finally, $\operatorname{disc}(D,H)\ge\sqrt{c_2}\,
	\bigl(p(1-p)q(1-q)\bigr)^3n^{3/2}$. This proves Theorem~\ref{thm:moderate-density}.
\end{proof}
\section{Intersections of several graphs}\label{sec:somegraphs}
The following lemma reduces the discrepancy of several graphs to that
of any two members of the family. It applies to both simple undirected
graphs and oriented graphs; tournaments are included in the latter case.
\begin{lemma}\label{lem:reduction-to-two-graphs}
	Let $r\ge2$ and $n\ge2$, and let $G_1,\ldots,G_r$ be either all simple
	undirected graphs or all oriented graphs of order $n$. Write $e(G_k)=p_k\binom{n}{2}$ for $1\le k\le r$ and set
	\begin{equation*}
		\kappa=
	\begin{cases}
		1, & \text{in the undirected case},\\
		2, & \text{in the oriented case}.
	\end{cases}
	\end{equation*}
	Then, for every pair $1\le i<j\le r$,
	\begin{equation*}
		\operatorname{disc}_r^+(G_1,\ldots,G_r)
	\operatorname{disc}_r^-(G_1,\ldots,G_r)\ge
	\left(
	\prod_{\ell\ne i,j}\frac{p_\ell}{\kappa}
	\right)^2
	\operatorname{disc}^+(G_i,G_j)
	\operatorname{disc}^-(G_i,G_j).
	\end{equation*}
\end{lemma}
\begin{proof}[\bf Proof]
	Choose $\pi_1,\ldots,\pi_r$ independently and uniformly from $S_n$,
	and write
	\begin{equation*}
	F_r(\boldsymbol{\pi})=\left|\bigcap_{k=1}^r E_{\pi_k}(G_k)
	\right|-\frac{1}{\kappa^{r-1}}\binom{n}{2}\prod_{k=1}^r p_k.
	\end{equation*}
	Fix $1\le i<j\le r$. Conditional on $\pi_i,\pi_j$, the remaining
	relabellings are still independent and uniform. For every fixed
	edge or directed edge $e$ and every $\ell\ne i,j$, $\mathbb P\bigl(e\in E_{\pi_\ell}(G_\ell)\bigr)={e(G_\ell)}/{\left(\kappa\binom{n}{2}\right)}={p_\ell}/{\kappa}$.
	Therefore
	\begin{align}\label{eq:conditional-two-graph-reduction}
		\mathbb E_{\boldsymbol{\pi}}\!\left(
	F_r(\boldsymbol{\pi})
	\,\middle|\,
	\pi_i,\pi_j
	\right)&=\mathbb{E}_{\boldsymbol{\pi}} \left( \left|\bigcap_{k=1}^r E_{\pi_k}(G_k)
	\right|\middle| \pi_i, \pi_j \right)-
	\frac{1}{\kappa^{r-1}}
	\binom{n}{2}\prod_{k=1}^r p_k\notag\\&=	\sum_{e\in E_{\pi_i}(G_i)\cap E_{\pi_j}(G_j)}
	\prod_{\ell\ne i,j}
	\mathbb P\bigl(e\in E_{\pi_\ell}(G_\ell)\bigr)-
	\frac{1}{\kappa^{r-1}}
	\binom{n}{2}\prod_{k=1}^r p_k\notag\\&=
	\left(\prod_{\ell\ne i,j}\frac{p_\ell}{\kappa}
	\right)\left|E_{\pi_i}(G_i)\cap E_{\pi_j}(G_j)\right|-\left(
	\prod_{\ell\ne i,j}\frac{p_\ell}{\kappa}
	\right)\frac{p_ip_j}{\kappa}\binom{n}{2}\notag\\&=
	\left(\prod_{\ell\ne i,j}\frac{p_\ell}{\kappa}\right)
	\left(\left|E_{\pi_i\pi_j^{-1}}(G_i)\cap E(G_j)\right|-\frac{p_ip_j}{\kappa}\binom{n}{2}
	\right).
	\end{align}
	By the definitions of the positive and negative discrepancies, we have $-\operatorname{disc}_r^-(G_1,\ldots,G_r)\le F_r(\boldsymbol{\pi})\le\operatorname{disc}_r^+(G_1,\ldots,G_r)$.
	Taking conditional expectations preserves these inequalities.
	Maximizing the conditional expectation over $\pi_i,\pi_j$ and using
	\eqref{eq:conditional-two-graph-reduction}, we obtain
	\begin{equation*}
	\operatorname{disc}_r^+(G_1,\ldots,G_r)\ge
	\max_{\pi_i,\pi_j}
	\mathbb E_{\boldsymbol{\pi}}\!\left(F_r(\boldsymbol{\pi})\,\middle|\,\pi_i,\pi_j
	\right)=\left(\prod_{\ell\ne i,j}\frac{p_\ell}{\kappa}\right)
	\operatorname{disc}^+(G_i,G_j).
	\end{equation*}
	Similarly,
	\begin{equation*}
	\operatorname{disc}_r^-(G_1,\ldots,G_r)\ge-\min_{\pi_i,\pi_j}\mathbb E_{\boldsymbol{\pi}}\!\left(F_r(\boldsymbol{\pi})\,\middle|\,\pi_i,\pi_j
	\right)=
	\left(\prod_{\ell\ne i,j}\frac{p_\ell}{\kappa}\right)
	\operatorname{disc}^-(G_i,G_j).
	\end{equation*}
	 Multiplying the two inequalities proves the product bound.
\end{proof}
\begin{proof}[\bf Proof of Theorem \ref{thm:somegraphs}]
Fix indices $1\le i<j\le r$ satisfying ${16}/{n}\le p_i,p_j\le1-{16}/{n}$. By Lemma~\ref{lem:reduction-to-two-graphs}, with $\kappa=1$, and
Theorem~\ref{thm:jgt2011}, we have
\begin{align*}
	\operatorname{disc}_r^+(G_1,\ldots,G_r)
	\operatorname{disc}_r^-(G_1,\ldots,G_r)&\ge
	\left(\prod_{\ell\ne i,j}p_\ell\right)^2
	\operatorname{disc}^+(G_i,G_j)
	\operatorname{disc}^-(G_i,G_j)\\&\ge
	\frac{n^3}{10^{20}}
	\left(\prod_{\ell\ne i,j}p_\ell\right)^2
	\bigl(p_i(1-p_i)p_j(1-p_j)\bigr)^4.
\end{align*}
Taking the maximum over all pairs $i<j$ satisfying the density
condition proves the theorem.
\end{proof}
\begin{proof}[\bf Proof of Theorem~\ref{thm:sometournaments}]
	Since each $T_\ell$ is a tournament, we have $p_\ell=1$.
	Applying Lemma~\ref{lem:reduction-to-two-graphs} with $\kappa=2$
	to any pair $i<j$, and then using Theorem~\ref{thm:tournament-intro}, gives
	\begin{equation*}
	\operatorname{disc}_r^+(T_1,\ldots,T_r)
	\operatorname{disc}_r^-(T_1,\ldots,T_r)\ge
	\left(
	\prod_{\ell\ne i,j}\frac12
	\right)^2
	\operatorname{disc}^+(T_i,T_j)
	\operatorname{disc}^-(T_i,T_j)\ge
	\frac{c_1}{4^{r-2}}\,n^3.
	\end{equation*}
	This proves the theorem.
\end{proof}
\begin{proof}[\bf Proof of Theorem~\ref{thm:someoriented}]
	Fix indices $1\le i<j\le r$ satisfying ${16}/{n}\le p_i,p_j\le1-{16}/{n}$. By Lemma~\ref{lem:reduction-to-two-graphs}, with $\kappa=2$, and
	Theorem~\ref{thm:moderate-density},
	\begin{align*}
	\operatorname{disc}_r^+(D_1,\ldots,D_r)
	\operatorname{disc}_r^-(D_1,\ldots,D_r)&\ge
	\left(
	\prod_{\ell\ne i,j}\frac{p_\ell}{2}
	\right)^2
	\operatorname{disc}^+(D_i,D_j)
	\operatorname{disc}^-(D_i,D_j)\\
	&\ge
	c_2n^3
	\left(
	\prod_{\ell\ne i,j}\frac{p_\ell}{2}
	\right)^2
	\bigl(p_i(1-p_i)p_j(1-p_j)\bigr)^6.
	\end{align*}
	Taking the maximum over all pairs $i<j$ satisfying the density
	condition proves the theorem.
\end{proof}
\section*{Declaration on the use of AI}
The authors used ChatGPT 5.6 Pro to assist in discussing proof strategies, checking proofs, and improving exposition.

\end{document}